\documentclass[11pt,reqno]{amsart}
\usepackage{hyperref}
\usepackage{amssymb,amsmath,amsfonts,latexsym}
\usepackage{bm}
\usepackage{mathtools}
\usepackage{enumerate}
\usepackage{enumitem}
\usepackage{mathrsfs}
\usepackage{array,graphics,color}
\usepackage{csquotes}
\allowdisplaybreaks

\numberwithin{equation}{section}
\newtheorem{dfn}{Definition}[section]
\newtheorem{thm}[dfn]{Theorem}
\newtheorem{lma}[dfn]{Lemma}

\newtheorem{ppsn}[dfn]{Proposition}

\newtheorem{xmpl}[dfn]{Example}
\newtheorem{rmrk}[dfn]{Remark}
\newtheorem{prop}[dfn]{Property}
\newtheorem{note}[dfn]{Note}

\DeclarePairedDelimiterX{\norm}[1]{\lVert}{\rVert}{#1}
\DeclarePairedDelimiterX{\bnorm}[1]{\big\lVert}{\big\rVert}{#1}
\DeclarePairedDelimiterX{\Bnorm}[1]{\Big\lVert}{\Big\rVert}{#1}

\begin{document}
	
\title[On density topology using ideals in the space of reals]{On density topology using ideals in the space of reals}

\author[Banerjee] {Amar Kumar Banerjee}
\address{Department of Mathematics, The University of Burdwan, Burdwan-713104, West Bengal, India}
\email{akbanerjee1971@gmail.com}
	
\author[Debnath] {Indrajit Debnath}
\address{Department of Mathematics, The University of Burdwan, Burdwan-713104, West Bengal, India}
\email{ind31math@gmail.com}

\subjclass[2020]{26E99, 54C30, 40A35}
	
\keywords{Density topology, ideal, $\mathcal{I}$-density topology}
	
\begin{abstract}
In this paper we have introduced the notion of $\mathcal{I}$-density topology in the space of reals introducing the notions of upper $\mathcal{I}$-density and lower $\mathcal{I}$-density where $\mathcal{I}$ is an ideal of subsets of the set of natural numbers. We have further studied certain separation axioms of this topology. 
\end{abstract}
\maketitle
	
\section{Introduction and Preliminaries}

The idea of convergence of real sequences was generalized to the notion of statistical convergence in \cite{Fast,Schoenberg}. For $K \subset \mathbb{N}$, the set of natural numbers and $n \in \mathbb{N}$ let $K_n =\{k\in K: k \leq n\}$. The natural density of the set $K$ is defined by $d(K)=\lim_{n \rightarrow{\infty}}\frac{|K_n|}{n}$, provided the limit exists, where $|K_n|$ stands for the cardinality of the set $K_n$. A sequence $\{x_n\}_{n \in \mathbb{N}}$ of real numbers is said to be statistically convergent to $x_0$ if for each $\epsilon > 0$ the set $K(\epsilon) = \{k\in \mathbb{N} : |x_k - x_0| \geq \epsilon\}$ has natural density zero.

After this pioneering work, the theory of statistical convergence of real sequences were generalized to the idea of $\mathcal{I}$-convergence of real sequences by P. Kostyrko et al. \cite{Kostyrko 2000} using the notion of ideal  $\mathcal{I}$ of subsets of $\mathbb{N}$, the set of natural numbers. A subcollection  $\mathcal{I} \subset 2^\mathbb{N}$ is called an ideal if $A,B \in  \mathcal{I}$ implies $A \cup B \in  \mathcal{I}$ and $A\in  \mathcal{I}, B\subset A$ imply $B\in \mathcal{I}$. $\mathcal{I}$ is called nontrivial ideal if $\mathcal{I} \neq \{\phi\}$ and $\mathbb{N}\notin \mathcal{I}$. $\mathcal{I}$ is called admissible if it contains all the singletons. It is easy to verify that the family $\mathcal{I}_d=\{A \subset \mathbb{N}: d(A)=0\}$ forms a non-trivial admissible ideal of subsets of $\mathbb{N}$. If $\mathcal{I}$ is a proper non-trivial ideal then the family of sets $\mathcal{F}(\mathcal{I}) = \{M\subset \mathbb{N} : \mathbb{N}\setminus M \in \mathcal{I}\}$ is a filter on $\mathbb{N}$ and it is called the filter associated with the ideal $\mathcal{I}$ of $\mathbb{N}$.

A sequence $\{x_n\}_{n \in \mathbb{N}}$ of real numbers is said to be $\mathcal{I}$-convergent \cite{Kostyrko 2000} to $x_0$ if the set $K(\epsilon) = \{k\in \mathbb{N} : |x_k - x_0| \geq \epsilon\}$ belongs to $\mathcal{I}$ for each $\epsilon>0$. A sequence $\{x_n\}_{n \in \mathbb{N}}$ of real numbers is said to be $\mathcal{I}$-bounded if there is a real number $M>0$ such that $\{k \in \mathbb{N}:|x_k|>M\}\in \mathcal{I}$. Further many works were carried out in this direction by many authors \cite{banerjee 2,banerjee 3,Lahiri 2005}.

Demirci \cite{Demirci} introduced the notion of  $\mathcal{I}$-limit superior and inferior of real sequence and proved several basic properties.

Let $\mathcal{I}$ be an admissible ideal in $\mathbb{N}$ and $x=\{x_n\}_{n \in \mathbb{N}}$ be a real sequence. Let,
$B_x =\{b\in \mathbb{R} : \{k:x_k > b\}\notin \mathcal{I}\}$ and $A_x =\{a\in \mathbb{R} : \{k:x_k < a\}\notin \mathcal{I}\}$. Then the $\mathcal{I}$-limit superior of $x$ is given by,
\begin{equation*}
    \mathcal{I}-\limsup  x  = \left\{
     \begin{array}{lr}
       \sup B_x & \text{if}\ B_x \neq \phi\\
       - \infty & \text{if}\ B_x = \phi
      
     \end{array} 
     \right.
\end{equation*}

and the $\mathcal{I}$-limit inferior of $x$ is given by,
\begin{equation*}
  \mathcal{I}-\liminf  x = \left\{
     \begin{array}{lr}
       \inf A_x & \text{if}\ A_x \neq \phi\\
       \infty & \text{if}\ A_x = \phi
      
     \end{array} 
     \right.
\end{equation*}
Further Lahiri and Das \cite{Lahiri 2003} carried out some more works in this direction. Throughout the paper the ideal $\mathcal{I}$ will always stand for a nontrivial admissible ideal of subsets of $\mathbb{N}$.

We shall use the notation $m^\star$ for the outer Lebesgue measure, $m_{\star}$ for the inner Lebesgue measure, $\mathcal{L}$ for the $\sigma$-algebra of Lebesgue measurable sets and $m$ for the Lebesgue measure. Throughout $\mathbb{R}$ stands for the set of all real numbers. The symbol $\mathfrak{T}_U$ stands for the natural topology on $\mathbb{R}$. Wherever we write $\mathbb{R}$ it means that $\mathbb{R}$ is equipped with natural topology unless otherwise stated. By \enquote{Euclidean $F_{\sigma}$ and Euclidean $G_{\delta}$ set} we mean $F_{\sigma}$ and $G_{\delta}$ set in $\mathbb{R}$ equipped with natural topology. The symmetric difference of two sets $A$ and $B$ is $(A-B)\cup (B-A)$ and it is denoted by $A \triangle B$. For $x \in \mathbb{R}$ and $A \subset \mathbb{R}$ we define $dist(x,A)=\inf \{|x-a|:a \in A\}$. By \enquote{a sequence of intervals $\{J_n\}_{n \in \mathbb{N}}$ about a point $p$} we mean $p \in \bigcap_{n \in \mathbb{N}}J_n$.

The idea of density functions and the corresponding density topology were studied in several spaces like the space of real numbers \cite{Riesz}, Euclidean $n$-space \cite{Troyer}, metric spaces \cite{Lahiri 1998} etc. Goffman et al. \cite{Goffman 1961} and H.E. White \cite{White} studied further on some properties of density topology on real numbers. 

For, $E\in \mathcal{L}$ and $x\in \mathbb{R}$ the upper density of $E$ at the point $x$ denoted by $d^-(x,E)$ and the lower density of $E$ at the point $x$ denoted by $d_-(x,E)$ are defined in \cite{White} as follows:
\begin{equation*}
    d^-(x,E)= \limsup_{n \to \infty}\{\frac{m(E \cap I)}{m(I)}:I  \ \textrm{is a closed interval},  x\in I, 0<m(I)<\frac{1}{n}\}
\end{equation*}

\begin{equation*}
    d_-(x,E)= \liminf_{n \to \infty}\{\frac{m(E \cap I)}{m(I)}:I \ \textrm{is a closed interval},  x\in I, 0<m(I)<\frac{1}{n}\}
\end{equation*}
If $d_-(x,E)=d^-(x,E)=\gamma$ we say $E$ has density $\gamma$ at the point $x$ and denote $\gamma$ by $d(x,E)$. Moreover $x\in \mathbb{R}$ is a density point of $E$ if and only if $d(x,E)=1$. Let us take the family
\begin{center}
    $\mathfrak{T}_d=\{E \in \mathcal{L}:d(x,E)=1 \  \mbox{for all} \ x \in E\}$
\end{center}
Then $\mathfrak{T}_d$ is ordinary density topology on $\mathbb{R}$ \cite{Goffman 1961} and it is finer than the usual topology $\mathfrak{T}_U$. Any member of $\mathfrak{T}_d$ is called $d$-open set.

In the recent past the notion of classical Lebesgue density point has been generalised by weakening the assumptions on the sequences of intervals and consequently several notions like $\langle s \rangle$-density point by M. Filipczak and J. Hejduk \cite{Filipczak 2004}, $\mathcal{J}$-density point by J. Hejduk and R. Wiertelak  \cite{Hejduk 2014}, $\mathcal{S}$-density point by F. Strobin and R. Wiertelak \cite{Strobin} have been obtained.

In this paper we try to generalize the classical Lebesgue density point using the notion of ideal $\mathcal{I}$ of subsets of naturals. We have given the notion of $\mathcal{I}$-density topology in the space of reals introducing the notions of upper $\mathcal{I}$-density and lower $\mathcal{I}$-density. In Section \ref{section 3} we prove Lebesgue $\mathcal{I}$-density Theorem and in Section \ref{section 4} we have given $\mathcal{I}$-density topology on the real line. We have shown that $\mathcal{I}$-density topology is finer than the density topology on the real line. We have also studied the idea of $\mathcal{I}$-approximate continuity and it is proved that $\mathcal{I}$-approximately continuous functions are indeed continuous if the real number space is endowed with $\mathcal{I}$-density topology. The existence of bounded $\mathcal{I}$-approximately continuous functions has been given using Lusin-Menchoff condition for $\mathcal{I}$-density. In the last section we prove that $\mathcal{I}$-density topology is completely regular but not normal.

\section{$\mathcal{I}$-density}\label{section 2}

\begin{dfn}
For $E\in \mathcal{L}$, $p\in \mathbb{R}$ and $n \in \mathbb{N}$ the upper $\mathcal{I}$-density of $E$ at the point $p$ denoted by $\mathcal{I}-d^-(p,E)$ and the lower $\mathcal{I}$-density of $E$ at the point $p$ denoted by $\mathcal{I}-d_-(p,E)$ are defined as follows: Suppose $\{J_n\}_{n \in \mathbb{N}}$ be a sequence of closed intervals about $p$ such that
\begin{center}
    $\mathscr{S}(J_n)=\{n\in \mathbb{N}:0<m(J_n)<\frac{1}{n}\} \in \mathcal{F}(\mathcal{I})$
\end{center} Take
\begin{equation*}
 x_n = \frac{m(J_n \cap E)}{m(J_n)} 
\end{equation*}
Then $x=\{x_n\}_{n \in \mathbb{N}}$ is a sequence of non-negative real numbers. Now if $B_x =\{b\in \mathbb{R} : \{k:x_k > b\}\notin \mathcal{I}\}$ and $A_x =\{a\in \mathbb{R} : \{k:x_k < a\}\notin \mathcal{I}\}$ we define,\\
\begin{equation*}
 \mathcal{I}-d^-(p,E) = \sup B_x  
\end{equation*}
and
\begin{equation*}
 \mathcal{I}-d_-(p,E)=\inf A_x  
\end{equation*}

Now, if $\mathcal{I}-d_-(p,E)=\mathcal{I}-d^-(p,E)$ then we denote the common value by $\mathcal{I}-d(p,E)$ which we call as $\mathcal{I}$-density of $E$ at the point $p$ and clearly $\mathcal{I}-d(p,E)=\mathcal{I}-\lim x_n$.
\end{dfn}

A point $p_0 \in \mathbb{R}$ is an $\mathcal{I}$-\textit{density point} of $E\in \mathcal{L}$ if $\mathcal{I}-d(p_0,E)=1$. 

If a point $p_0\in \mathbb{R}$ is an $\mathcal{I}$-density point of the set $\mathbb{R}\setminus E$, then $p_0$ is an $\mathcal{I}$-\textit{dispersion point} of $E$.

\begin{rmrk}
The notion of $\mathcal{I}$-density point is more general than the notion of density point, i.e. a point in $\mathbb{R}$ may be an $\mathcal{I}$-density point of a set in $\mathbb{R}$ without being a density point which is shown in the following example.
\end{rmrk}

\begin{xmpl}
Let us consider the ideal $\mathcal{I}_d$ of subsets of $\mathbb{N}$ where $\mathcal{I}_d$ is the ideal containing all those subsets of $\mathbb{N}$ whose natural density is zero. Now, for any point $x \in \mathbb{R}$ consider the following collections:
\begin{center}
    $J_x=\{\{J_n\}_{n \in \mathbb{N}}: \{J_n\} \textrm{is a sequence of closed intervals about} \ x \ \textrm{such that} \  \mathscr{S}(J_n)=\mathbb{N} \}$ and 
\end{center}
\begin{center}
    $G_x=\{\{J_n\}_{n \in \mathbb{N}}: \{J_n\} \textrm{is a sequence of closed intervals about} \ x \ \textrm{such that} \  \mathscr{S}(J_n)\in \mathcal{F}(\mathcal{I}_d) \}$
\end{center}
We claim that $J_x \subsetneqq G_x$. It is easy to observe that $J_x \subseteq G_x$, since $\mathbb{N} \in \mathcal{F}(\mathcal{I}_d)$.\\ Now in particular let us take the following sequence $\{K_n\}_{n \in \mathbb{N}}$ of closed intervals about a point $x$.
\begin{align*}
K_n=[x-\frac{1}{2n+1},x+\frac{1}{2n+1}] \ \textrm{for} \ n \neq m^2 \ \textrm{where} \ m \in \mathbb{N} \\
    K_n=[x-n,x+n]  \qquad \ \ \ \textrm{for} \ n=m^2 \ \textrm{where} \ m \in \mathbb{N}\\
\end{align*}

We observe that for $n \neq m^2$, $m(K_n)= \frac{2}{2n+1} < \frac{1}{n}$ and for $n=m^2$, $m(K_{n})=2n \nless \frac{1}{n}$. Therefore, $\mathscr{S}(K_n)=\{n: n \neq m^2\} \in \mathcal{F}(\mathcal{I}_d)$. Since $\mathscr{S}(K_n) \neq \mathbb{N}$, $\{K_n\}_{n \in \mathbb{N}} \in G_x \setminus J_x$ and so $J_x \subsetneqq G_x$.\\

Let us take the set $E=(-1,1)$ and the point $x=0$. Let $\{K_n\}_{n \in \mathbb{N}} \in G_0 \setminus J_0$ be taken as above. Now if $x_n=\frac{m(K_n \cap E)}{m(K_n)}$ then

\begin{equation*}
    x_n =
\left\{
	\begin{array}{ll}
		1  & \mbox{if } n \neq m^2\\
		\frac{1}{m^2} & \mbox{if } n=m^2
	\end{array}
\right.
\end{equation*}
Therefore, since the subsequence $\{x_{n}\}_{n=m^2}$ converges to $0$ and the subsequence $\{x_n\}_{n \neq m^2}$ converges to $1$, $\lim_{n} x_n$ does not exists but $\mathcal{I}-\lim_{n} x_n=1$. Since $\lim_n x_n$ does not exists so $0$ is not a density point of $E$. But given any $\{J_n\}_{n \in \mathbb{N}} \in G_0$ we have $\mathscr{S}(J_n)\in \mathcal{F}(\mathcal{I}_d)$. Also we observe that
\begin{center}
    $\{n:J_n \subset E\} \supset \mathscr{S}(J_n)$
\end{center}
Therefore, $\{n:J_n \subset E\} \in \mathcal{F}(\mathcal{I}_d)$. Now $J_n \subset E$ implies $x_n=\frac{m(J_n \cap E)}{m(J_n)}=1$. Thus, $\{n:J_n \subset E\} \subset \{n:x_n=1\}$. So, $\{n:x_n=1\} \in \mathcal{F}(\mathcal{I}_d)$. Therefore, $B_x=(-\infty,1)$ and $A_x=(1,\infty)$ and so,
\begin{center}
    $\mathcal{I}_d-d^{-}(0,E)=\sup B_x=1$ and $\mathcal{I}_d-d_{-}(0,E)=\inf A_x=1$.
\end{center}
Hence $\mathcal{I}_d-d(0,E)$ exists and equals to $1$. So, 0 is an $\mathcal{I}_d$-density point of $E$.
\end{xmpl}

\begin{note}
In particular if $\mathcal{I}=\mathcal{I}_{f}$ where $\mathcal{I}_{f}$ is the class of all finite subsets of $\mathbb{N}$ then our definition of $\mathcal{I}$-density coincides with usual definition of density.
\end{note}

The following theorem was given by K. Demirci \cite{Demirci}.
\begin{thm}{For any real sequence $x$, $\mathcal{I}-\liminf x \leq \mathcal{I}-\limsup x$.}
\end{thm}

Here we are proving some important results which will be needed later in our discussion.

\begin{thm} For any Lebesgue measurable set $A \subset \mathbb{R}$ and any point $x \in \mathbb{R}$,
\begin{center}
  $\mathcal{I}-d_-(x,A) \leq \mathcal{I}-d^{-}(x,A)$
\end{center}
\end{thm}
\begin{proof}
Let $\{I_k\}_{k \in \mathbb{N}}$ be a sequence of closed intervals about the point $x$ such that $\mathscr{S}(I_k) \in \mathcal{F(\mathcal{I})}$. Now let us take the real sequence $x_n = \frac{m(A \cap I_n)}{m(I_n)}$. Then,
\begin{equation*}
    \mathcal{I}-d_-(x,A)=\mathcal{I}-\liminf x_n \leq \mathcal{I}-\limsup x_n =\mathcal{I}-d^-(x,A)
\end{equation*}
\end{proof}

The following theorem is useful to prove our next results.

\begin{thm}[\cite{Lahiri 2003}]If $x=\{x_n\}_{n \in \mathbb{N}}$ and $y=\{y_n\}_{n \in \mathbb{N}}$ are two $\mathcal{I}$-bounded real number sequences, then
\end{thm}

 \begin{enumerate}[label=(\roman*)]
    \item $\mathcal{I}-\limsup (x+y) \leq \mathcal{I}-\limsup x +\mathcal{I}-\limsup y$
    \item $\mathcal{I}-\liminf (x+y) \geq \mathcal{I}-\liminf x + \mathcal{I}-\liminf y$
\end{enumerate}

\begin{ppsn} Given an $\mathcal{I}$-bounded real sequence $\{x_n\}_{n \in \mathbb{N}}$ and a real number $c$, 
\end{ppsn}
\begin{center}
\begin{enumerate}[label=(\roman*)]
    \item $\mathcal{I}-\liminf (c+x_n)=c+\mathcal{I}-\liminf x_n$
    \item $\mathcal{I}-\limsup (c+x_n)=c+\mathcal{I}-\limsup x_n$
\end{enumerate}
\end{center}

\begin{proof}
$(i)$ It is obvious that $\mathcal{I}-\liminf (c+x_n) \geq c+\mathcal{I}-\liminf x_n$. Now we are to show $\mathcal{I}-\liminf (c+x_n) \leq c+\mathcal{I}-\liminf x_n$. Let $y_n=c+x_n$. Then $\mathcal{I}-\liminf x_n = \mathcal{I}-\liminf (y_n-c) \geq \mathcal{I}-\liminf y_n -c$. Therefore, $\mathcal{I}-\liminf y_n \leq c+\mathcal{I}-\liminf x_n$. So we can conclude that  $\mathcal{I}-\liminf(c+x_n)=c+\mathcal{I}-\liminf x_n$. The proof of $(ii)$ is analogous.
\end{proof}

\begin{ppsn}
{For any real sequence $x=\{x_n\}_{n \in \mathbb{N}}$,
}
\end{ppsn}
\begin{enumerate}[label=(\roman*)]
    \item $\mathcal{I}-\limsup (-x)=-(\mathcal{I}-\liminf x)$
    \item $\mathcal{I}-\liminf (-x)=-(\mathcal{I}-\limsup x)$
\end{enumerate}
\begin{proof}
$(i)$ Let us take $B_x =\{b\in \mathbb{R} : \{k:x_k > b\}\notin \mathcal{I}\}$ and $A_x =\{a\in \mathbb{R} : \{k:x_k < a\}\notin \mathcal{I}\}$. Then clearly, $B_{(-x)}= -A_{x}$.

Therefore, $\mathcal{I}-\limsup (-x)  = \sup B_{(-x)} = \sup (-A_{x}) = - \inf A_{x} = -\mathcal{I}-\liminf (x)$. In a similar manner we can prove $(ii)$.
\end{proof}

\begin{lma}
For any disjoint Lebesgue measurable subsets $A, B$ of $\mathbb{R}$ and any point $x \in \mathbb{R}$ if $\mathcal{I}-d (x,A)$ and $\mathcal{I}-d (x,B)$ exist, then $\mathcal{I}-d (x,A \cup B)$ exists and $\mathcal{I}-d (x,A \cup B)=\mathcal{I}-d (x,A)+\mathcal{I}-d (x,B)$.
\end{lma}
\begin{proof}Let $\{I_k\}_{k \in \mathbb{N}}$ be a sequence of closed intervals about the point $x$ such that $\mathscr{S}(I_k) \in \mathcal{F(\mathcal{I})}$. Now let us take the real sequences $\{x_n\}_{n \in \mathbb{N}},\{y_n\}_{n \in \mathbb{N}},\{z_n\}_{n \in \mathbb{N}}$ defined as $x_n = \frac{m(A \cap I_n)}{m(I_n)}$, $y_n = \frac{m(B \cap I_n)}{m(I_n)}$ and $z_n = \frac{m((A \cup B) \cap I_n)}{m(I_n)}$. Then each of $\{x_n\}_{n \in \mathbb{N}},\{y_n\}_{n \in \mathbb{N}},\{z_n\}_{n \in \mathbb{N}}$ is bounded hence $\mathcal{I}$-bounded. Since $A, B$ are disjoint we have for any $n \in \mathscr{S}(I_k)$, $m((A\cup B)\cap I_n)=m(A\cap I_n)+m(B\cap I_n)$. So, $z_n= x_n + y_n$ for $n \in \mathscr{S}(I_k)$. Hence,
\begin{equation}
    \begin{split}
        \mathcal{I}-d^{-}(x,A \cup B) & = \mathcal{I}-\limsup z_n \\
        & = \mathcal{I}-\limsup(x_n + y_n)\\
        & \leq \mathcal{I}-\limsup x_n + \mathcal{I}-\limsup y_n\\
        & = \mathcal{I}-d^{-}(x,A)+ \mathcal{I}-d^{-}(x,B)\\
        & = \mathcal{I}-d_{-}(x,A)+ \mathcal{I}-d_{-}(x,B)\\
        & = \mathcal{I}-\liminf x_n + \mathcal{I}-\liminf y_n\\
        & \leq \mathcal{I}-\liminf (x_n+ y_n)\\
        & = \mathcal{I}-\liminf z_n\\
        & = \mathcal{I}-d_{-}(x,A \cup B)
    \end{split}
\end{equation}
Also, by Theorem 2.6, $\mathcal{I}-d^{-}(x,A \cup B) \geq \mathcal{I}-d_{-}(x,A \cup B)$. Therefore, $\mathcal{I}-d(x,A \cup B)$ exists and $\mathcal{I}-d^{-}(x,A \cup B) = \mathcal{I}-d_{-}(x,A \cup B) = \mathcal{I}-d(x,A \cup B)$. From (2.1) it is clear that $\mathcal{I}-d(x,A \cup B) \leq \mathcal{I}-d(x,A) +\mathcal{I}-d(x,B) \leq \mathcal{I}-d(x,A \cup B)$. Hence, $\mathcal{I}-d(x,A \cup B) = \mathcal{I}-d(x,A) +\mathcal{I}-d(x,B)$.
\end{proof}

\begin{lma}
If $\mathcal{I}-d(x,A)$ and $\mathcal{I}-d(x,B)$ exist and $A \subset B$. Then $\mathcal{I}-d(x,B \setminus A)$ exists and $\mathcal{I}-d(x,B \setminus A)=\mathcal{I}-d(x,B)-\mathcal{I}-d(x,A)$
\end{lma}
\begin{proof}Since $A$ and $B$ are measurable sets, for any sequence $\{I_k\}_{k \in \mathbb{N}}$ of closed intervals about the point $x$ such that $\mathscr{S}(I_k) \in \mathcal{F(\mathcal{I})}$ we have $m((B \setminus A) \cap I_n)=m(B \cap I_n)-m(A \cap I_n)$. Consider $x_n$ and $y_n$ as in previous lemma. Take $p_n=\frac{m((B \setminus A) \cap I_n)}{m(I_n)}$. So, $p_n = y_n - x_n$. It is easy to see $\{p_n\}_{n \in \mathbb{N}}$ is bounded and hence an $\mathcal{I}$-bounded sequence. So, 
\begin{align*}
    \mathcal{I}-d_{-}(x,B \setminus A)  & = \mathcal{I}-\liminf p_n \\
         & = \mathcal{I}-\liminf(y_n - x_n)\\
         & \geq \mathcal{I}-\liminf y_n - \mathcal{I}-\limsup x_n\\
        & = \mathcal{I}-d_{-}(x,B)- \mathcal{I}-d^{-}(x,A)\\
         & = \mathcal{I}-d^{-}(x,B)- \mathcal{I}-d_{-}(x,A)\\
         & = \mathcal{I}-\limsup y_n - \mathcal{I}-\liminf x_n \\
         & \geq \mathcal{I}-\limsup (y_n - x_n)\\
        & = \mathcal{I}-\limsup p_n\\
        & = \mathcal{I}-d^{-}(x,B \setminus A)
\end{align*}
Therefore, $\mathcal{I}-d(x,B \setminus A)$ exists and $\mathcal{I}-d^{-}(x,B \setminus A)=\mathcal{I}-d_{-}(x,B \setminus A)=\mathcal{I}-d(x,B \setminus A)$. So,  $\mathcal{I}-d(x,B \setminus A) \geq \mathcal{I}-d(x,B) - \mathcal{I}-d(x,A) \geq \mathcal{I}-d(x,B \setminus A)$. Hence, $\mathcal{I}-d(x,B \setminus A)= \mathcal{I}-d(x,B)-\mathcal{I}-d(x,A).$
\end{proof}

\begin{thm}
For any measurable set $H$, $\mathcal{I}$-density of $H$ at a point $p$ exists if and only if $\mathcal{I}-d^{-}(p,H)+\mathcal{I}-d^{-}(p,H^{c})=1$.
\end{thm}
\begin{proof} Let $\{I_k\}_{k \in \mathbb{N}}$ be a sequence of closed intervals about the point $p$ such that $\mathscr{S}(I_k) \in \mathcal{F(\mathcal{I})}$. Let $x_n=\frac{m(I_n \cap H)}{m(I_n)}$ and $y_n=\frac{m(I_n \cap H^{c})}{m(I_n)}$. Then $x_n + y_n =1$ $\forall n \in \mathscr{S}(I_k)$. Both $\{x_n\}_{n \in \mathbb{N}}$ and $\{y_n\}_{n \in \mathbb{N}}$ are $\mathcal{I}$-bounded sequences.\\
Necessary part: Let $\mathcal{I}$-density of a measurable set $H$ at the point $p$ exists. Now 
\begin{align*}
    \mathcal{I}-d^{-}(p,H) & = \mathcal{I}-d_{-}(p,H) = \mathcal{I}-\liminf x_n  = \mathcal{I}-\liminf (1-y_n)
         \\ & = 1- \mathcal{I}-\limsup y_n  = 1- \mathcal{I}-d^{-}(p,H^{c})
\end{align*}

Sufficient part: Let $\mathcal{I}-d^{-}(p,H)+\mathcal{I}-d^{-}(p,H^{c})=1$. Then,
\begin{align*}
  \mathcal{I}-d^{-}(p,H)  & = 1-\mathcal{I}-d^{-}(p,H^{c})
         = 1- \mathcal{I}-\limsup y_n
          = 1+ \mathcal{I}-\liminf (-y_n)\\
         & = \mathcal{I}-\liminf (1- y_n)
          = \mathcal{I}-\liminf x_n\\
         & = \mathcal{I}-d_{-}(p,H)  
\end{align*}
Hence, $\mathcal{I}$-density of $H$ at $p$ exists.
\end{proof}

\section{Lebesgue $\mathcal{I}-$density theorem}\label{section 3}
Let $H \subset \mathbb{R}$ be a measurable set. Let us denote the set of points of $\mathbb{R}$ at which $H$ has $\mathcal{I}$-density 1 by $\Theta_\mathcal{I}(H)$.

\begin{thm}
For any measurable set $H \subset \mathbb{R}$, $\Theta_\mathcal{I}(H) - H \subset H^{c}- \Theta_\mathcal{I}(H^{c})$ 
\end{thm}
\begin{proof}It is obvious that $\Theta_\mathcal{I}(H) - H \subset H^{c}$. Now we show if  $x \in \Theta_\mathcal{I}(H)$ then $x \notin \Theta_\mathcal{I}(H^{c})$. Suppose if possible, $x \in \Theta_\mathcal{I}(H) \cap \Theta_\mathcal{I}(H^{c})$. Then $\mathcal{I}-d(x,H)=1$ and $\mathcal{I}-d(x,H^{c})=1$. But this leads to a contradiction to theorem 2.12. Therefore, $\Theta_\mathcal{I}(H) \cap \Theta_\mathcal{I}(H^{c})$ is an empty set. Thus, $\Theta_\mathcal{I}(H) - H \subset H^{c}- \Theta_\mathcal{I}(H^{c})$.
\end{proof}

Here we prove an analogue of classical Lebesgue density theorem by the idea presented in \cite{Oxtoby} (Theorem 3.20).
\begin{thm}
For any measurable set $H \subset \mathbb{R},$
\begin{center}
 $m(H \triangle \Theta_\mathcal{I}(H))=0$ 
\end{center}

\end{thm}
\begin{proof}
It is sufficient to show that $H-\Theta_\mathcal{I}(H)$ is a null set, since $\Theta_\mathcal{I}(H)-H \subset H^{c}-\Theta_\mathcal{I}(H^{c})$ and $H^{c}$ is measurable. Let us assume without any loss of generality $H$ is bounded because if $H$ is unbounded it can be written as $\bigcup_{n=1}^{\infty} H_{n}$ where each $H_n$ is bounded. For $\mu >0$ let us take
\begin{align}
    C_{\mu}=\{x \in H: \mathcal{I}-d_{-}(x,H)<1-\mu\}.
\end{align} Then, $\mu_1 < \mu_2 \implies C_{\mu_2} \subset C_{\mu_1}$ and $H-\Theta_\mathcal{I}(H)=\bigcup_{\mu>0}C_{\mu}$. We are to show that $m^{\star}(C_\mu)=0$. Let, if possible $m^{\star}(C_\mu)> 0$ for some $\mu >0$. Since $C_\mu \subset H$ and $H$ is bounded so $C_\mu$ is bounded. Then there exists a bounded open set $G \supset C_\mu$ such that $(1-\mu)m(G)<m^{\star}(C_\mu)$. Let $\mathcal{F}$ be the family of all closed intervals $I$ such that $I \subset G$ and $m(H \cap I)\leq (1-\mu)m(I)$. Then for each $x \in C_{\mu}$ $\exists J \in \mathcal{F}$ such that $x \in J$ and $m(J)< \epsilon$ for arbitrary small $\epsilon >0$. So, $C_{\mu}$ is covered by $\mathcal{F}$ in the sense of Vitali. For any disjoint sequence $\{I_k\}_{k\in \mathbb{N}}$ of elements of $\mathcal{F}$, 
\begin{align*}
        m^{\star}(C_\mu \cap (\bigcup_{k \in \mathbb{N}}I_k)) & =
        m^{\star}(\bigcup_{k \in \mathbb{N}}(C_\mu \cap I_k) \leq \sum_{k \in \mathbb{N}} m^{\star}(C_\mu \cap I_k) \leq \sum_{k \in \mathbb{N}} m(H \cap I_k) \\ & \leq (1-\mu)\sum_{k \in \mathbb{N}} m(I_k) <  (1-\mu) m(G) < m^{\star}(C_{\mu})
\end{align*}

\begin{align}
    \therefore m^{\star}(C_\mu \setminus \bigcup_{k \in \mathbb{N}} I_k)>0 
\end{align}
We construct a disjoint sequence $\{J_k\}_{k \in \mathbb{N}}$ of elements in $\mathcal{F}$ as follows. Let $\alpha_0 = \sup_{J \in \mathcal{F}} m(J)$. Choose $J_1 \in \mathcal{F}$ such that $m(J_1)>\frac{\alpha_0}{2}$. Take $\mathcal{F}_1=\{J \in \mathcal{F}: J \cap J_1 = \phi\}$. Then $\mathcal{F}_1$ is nonempty since $m^{\star}(C_\mu \setminus J_1)>0$, by 3.2. Let $\alpha_1 = \sup_{J \in \mathcal{F}_1} m(J)$. Choose $J_2 \in \mathcal{F}_1$ such that $m(J_2)>\frac{\alpha_1}{2}$. Take $\mathcal{F}_2=\{J \in \mathcal{F}_1: J \cap J_2 = \phi\}$. Then $\mathcal{F}_2$ is nonempty, by 3.2. Likewise we choose $J_1, J_2,\dots, J_n$. By induction let us take $\mathcal{F}_n =\{J \in \mathcal{F}_{n-1}:J \cap J_n = \phi\}$. Then $\mathcal{F}_n$ is nonempty, by 3.2. Let $\alpha_{n}=\sup\{m(J): J \in \mathcal{F}_n\}$. Choose $J_{n+1} \in \mathcal{F}_n$ such that $m(J_{n+1})>\frac{\alpha_n}{2}$. Take $B=C_{\mu} \setminus \bigcup_{k \in \mathbb{N}} J_k$. Then, by 3.2, $m^{\star}(B)>0$. Since $J_k \subset G$ $\forall k \in \mathbb{N}$, it follows that $\bigcup_{k \in \mathbb{N}} J_k \subset G$. Thus $\sum_{k=1}^{\infty} m(J_k) \leq m(G) < \infty$. Therefore, $\exists$ $n_0 \in \mathbb{N}$ such that $\sum_{k=n_0 +1}^{\infty} m(J_k) < \frac{m^{\star}(B)}{4}$. For $k>n_0$ let $Q_k$ denote the interval concentric with $J_k$ such that $m(Q_k)=4m(J_k)$. Now, $\sum_{k=n_0 +1}^{\infty} m(Q_k)=4\sum_{k=n_0 +1}^{\infty} m(J_k) < m^{\star}(B)$. So, the family of intervals $\{Q_k\}_{k>n_0}$ does not cover $B$. 

Let us take $b \in B \setminus \bigcup_{k=n_0+1}^{\infty}Q_k$. Then, $b \in C_\mu \setminus \bigcup_{k=1}^{n_0}J_k$. Since, $\mathcal{F}$ is a Vitali cover of $C_\mu$, $\exists$ an interval $J \in \mathcal{F}_{n_0}$ such that $b \in J$ and $b$ is the center of $J$. Clearly for some $k>n_0$,  $J \cap J_k \neq \phi$. Because if $J \cap J_k = \phi$ $\forall k > n_0$  then, since $J \in \mathcal{F}_{n_0}$, $J \cap J_k = \phi$ for $k=1,2, \dots, n_0$. Hence, $J \cap J_k = \phi$ $\forall k \in \mathbb{N}$. Thus $J \in \mathcal{F}_n$ $\forall n \in \mathbb{N}$ which implies that $m(J) \leq \alpha_{n} < 2 m(J_{n+1}) \forall n \in \mathbb{N}$.
Again, since $\sum_{k=1}^{\infty} m(J_k) \leq m(G) < \infty$ given any $\epsilon >0$ $\exists k_0 \in \mathbb{N}$ such that $\sum_{k=k_0}^{\infty} m(J_k)< \epsilon$. But, $\sum_{k=k_0}^{\infty} m(J_k)> \sum_{k=k_0}^{\infty} (\frac{\alpha_{k-1}}{2})$. So we get a contradiction.

So, let $k_0$ be the least positive integer for which $J \cap J_{k_0} \neq \phi$. Then, $k_0>n_0$ and $J \in \mathcal{F}_{k_{0}-1}$. Therefore, $m(J) \leq \alpha_{k_0 -1}< 2 m(J_{k_0})= \frac{m(Q_{k_0})}{2}$. Now, for $b \in J$ and $J \cap J_{k_{0}} \neq \phi$ we have the following two cases. 
\begin{enumerate}
    \item if $b \in J_{k_{0}}$ then $b \in Q_{k_{0}}$
    \item if $b \notin J_{k_{0}}$ then also we claim $b \in  Q_{k_{0}}$. \\
    Since $b$ is the center of $J$ let us take $J=[b-\frac{m(J)}{2},b+\frac{m(J)}{2}]$. Let $x_{k_0}$ be the center of $J_{k_{0}}$. Then take $J_{k_{0}}=[x_{k_0}-\frac{m(J_{k_0})}{2},x_{k_0}
    +\frac{m(J_{k_0})}{2}]$. \\
    
    Consequently,  $Q_{k_{0}}=[x_{k_0}-2m(J_{k_0}),x_{k_0}
    +2m(J_{k_0})]$. 
    
    Let $x \in J \cap J_{k_0}$.Then, $|b-x| \leq \frac{m(J)}{2}$ and $|x-x_{k_0}| \leq \frac{m(J_{k_0})}{2}$. Hence,
    \begin{align*}
        |b-x_{k_0}| \leq |b-x|+|x-x_{k_0}| \leq \frac{m(J)}{2}+ \frac{m(J_{k_0})}{2}
         < m(J_{k_0})+\frac{m(J_{k_0})}{2}
         = \frac{3}{2} m(J_{k_0})
         < 2m(J_{k_0})
    \end{align*}
\end{enumerate}

Hence, $b \in Q_{k_0}$ which implies that $b \in \bigcup_{k=n_0 +1}^{\infty} Q_k$. This leads to a contradiction to our choice of $b$ in $B \setminus \bigcup_{k=n_0 +1}^{\infty} Q_k$. So, $m^{\star}(C_\mu)=0$ for each $\mu >0$. Therefore, $m(H-\Theta_{\mathcal{I}}(H))=0$.
\end{proof}
 
The statement of this theorem may also be read as almost all points of an arbitrary measurable set $H$ are points of $\mathcal{I}$-density for $H$.

\section{$\mathcal{I}$-density topology}\label{section 4}
\begin{dfn}
A measurable set $E \subset \mathbb{R}$ is $\mathcal{I}-d$ open iff  $\mathcal{I}-d_-(x,E)=1$ $\forall x \in E$
\end{dfn}
Let us take the collection $\mathfrak{T}_\mathcal{I}=\{A \subset \mathbb{R}:$ $A$ is $\mathcal{I}-d$ open $\}$

\begin{thm}
The collection $\mathfrak{T}_\mathcal{I}$ is a topology on $\mathbb{R}$.
\end{thm} 
\begin{proof}By voidness $\phi \in \mathfrak{T}_\mathcal{I}$. Since $\mathbb{R} \in \mathcal{L}$ so for $E=\mathbb{R}$ and any $r\in \mathbb{R}$ let $\{I_k\}_{k \in \mathbb{N}}$ be a sequence of closed intervals about $r$ such that $\mathscr{S}(I_k)\in \mathcal{F}(\mathcal{I})$. It is clear that $\mathbb{R} \cap I_k =I_k$ for all $k$. Therefore $x_k=\frac{m(\mathbb{R} \cap I_k)}{m(I_k)}=1$ for all $k \in \mathbb{N}$. Then
\begin{equation*}
    A_x =\{a\in \mathbb{R} : \{k:x_k < a\}\notin \mathcal{I}\}=(1,\infty)
\end{equation*}
Thus, $\mathcal{I}-d_-(r,\mathbb{R})=inf A_x=1 \   \forall r \in \mathbb{R}$. Therefore, $\mathbb{R}\in \mathfrak{T}_\mathcal{I}$. 

Let $\Lambda$ be an arbitrary indexing set and $\{A_\alpha\}_{\alpha \in \Lambda}$ be a collection of sets in $\mathfrak{T}_\mathcal{I}$. We are to show, $\bigcup_{\alpha \in \Lambda}A_\alpha \in \mathfrak{T}_\mathcal{I}$. Let us take $A=\bigcup_{\alpha \in \Lambda}A_\alpha$. Clearly $A_\alpha$ is measurable and $\mathcal{I}-d$ open for each $\alpha \in \Lambda$. First we have to show,  $\bigcup_{\alpha \in \Lambda}A_\alpha$ is measurable. 

Let $x \in A$ so $x \in A_{\alpha}$ for some $\alpha \in \Lambda$. Since $A_\alpha$ is $\mathcal{I}-d$ open so $\mathcal{I}-d(x,A_{\alpha})=1$. Therefore, there exists a sequence $\{I_{n}^{x}\}$ of closed interval about $x$ such that $\mathscr{S}(I_{n}^{x}) \in \mathcal{F}(\mathcal{I})$. Hence $\mathcal{I}-\lim \frac{m(A_\alpha \cap I_{n}^{x})}{m(I_{n}^{x})}=1$. It means that for $\epsilon >0$ there exists $n_0 \in \mathbb{N}$ such that $\forall$ $n \in \mathscr{S}(I_{n}^{x})$ and $n>n_0$ we have 
\begin{equation*}
    1-\epsilon<\frac{m(A_\alpha \cap I_{n}^{x})}{m(I_{n}^{x})}<1+\epsilon
\end{equation*}
So for some suitable $k$ we have $\frac{m(A_\alpha \cap I_{k}^{x})}{m(I_{k}^{x})}>1-\epsilon$. Since $A_\alpha$ is measurable so $A_\alpha \cap I_{k}^{x}$ is measurable subset of $A$. If $A$ is bounded, by Vitali Covering Theorem for $\mathbb{R}$, $A$ contains a measurable set $G$ such that $m^{\star}(A \setminus G)<\epsilon$ $m(G)$. Therefore, $A$ is measurable. If, $A$ is unbounded, then $A$ can be written as $A=\bigcup_{n=1}^{\infty} A_n$ where each $A_n$ is bounded and measurable. Therefore, $A$ is measurable. 

Now we show that for all $x \in A$, \  $\mathcal{I}-d_-(x,A)=1$. Since, $x \in A$, $x \in A_\alpha$ for some $\alpha$. So, $\mathcal{I}-d_-(x,A_\alpha)=1$. Since, $\mathcal{I}-d_-(x,A) \geq \mathcal{I}-d_-(x,A_\alpha)=1$. Therefore, $\mathcal{I}-d_-(x,A)=1$  $\forall x \in A$. Hence, $A=\bigcup_{\alpha \in \Lambda}A_\alpha \in \mathfrak{T}_\mathcal{I}$. 

Finally let $A,B \in \mathfrak{T}_\mathcal{I}$ we show $A \cap B \in \mathfrak{T}_\mathcal{I}$. Since both $A,B$ are measurable, $A \cap B$ is measurable. Now, for any $x \in A \cap B$ we show that $\mathcal{I}-d_-(x,A \cap  B)=1$. It is sufficient to show that $\mathcal{I}-d_-(x,A \cap  B)\geq 1$ $\forall x \in A \cap B$. Let $\{I_k\}_{k \in \mathbb{N}}$ be a sequence of intervals about a point $x$ such that $\mathscr{S}(I_k)\in \mathcal{F}(\mathcal{I})$. Then let us define $a_n=\frac{m(A \cap I_n)}{m(I_n)}$, $b_n=\frac{m(B \cap I_n)}{m(I_n)}$ and $p_n=\frac{m(A \cap B \cap I_n)}{m(I_n)}$. Then 
\begin{equation*}
    \frac{m(A \cap I_n)}{m(I_n)}+\frac{m(B \cap I_n)}{m(I_n)}\leq 1+\frac{m(A \cap B \cap I_n)}{m(I_n)}
\end{equation*} Then $a_n+b_n \leq 1+p_n$. Taking $\mathcal{I}-\liminf$ on both sides we have
\begin{align*}
    \mathcal{I}-d_-(x,A)+\mathcal{I}-d_-(x,B)
         \leq \mathcal{I}-\liminf \{a_n+b_n\} & \leq \mathcal{I}-\liminf \{1+p_n\}\\
        & = 1+ \mathcal{I}-\liminf p_n \\ & = 1+\mathcal{I}-d_-(x,A \cap B)
\end{align*}
Now since $A,B \in \mathfrak{T}_\mathcal{I}$ we have $\mathcal{I}-d_-(x,A \cap B) \geq 1$. So, $\mathfrak{T}_\mathcal{I}$ is a topology on $\mathbb{R}$.
\end{proof}

The topology $\mathfrak{T}_\mathcal{I}$ is called the $\mathcal{I}$-density topology on $\mathbb{R}$ and the pair $(\mathbb{R},\mathfrak{T}_\mathcal{I})$ is the corresponding topological space.

\begin{thm}
The family $\mathfrak{T}_{\mathcal{I}}$ is a topology on the real line finer than the density topology $\mathfrak{T}_d$.
\end{thm}
\begin{proof} For given any set $E$ in $\mathbb{R}$ if $x$ is a density point of $E$ then $d(x,E)=1$. So for any sequence of intervals $\{I_n\}_{n \in \mathbb{N}}$ about $x$ such that $0<m(I_n)<\frac{1}{n}$ for all $n \in \mathbb{N}$ we have $\lim_{n \rightarrow{\infty}} \frac{m(E \cap I_n)}{m(I_n)}=1$. Since $\mathcal{I}$ is an admissible ideal so $\mathcal{I}-\lim_{n \rightarrow{\infty}} \frac{m(E \cap I_n)}{m(I_n)}=1$ and $\mathscr{S}(I_n)=\mathbb{N}\in \mathcal{F}(\mathcal{I})$. Thus, $x$ is an $\mathcal{I}$-density point of $E$. 

Now for any $U \in \mathfrak{T}_d$ any point $x$ in $U$ is a density point of $U$. Thus any point of $U$ is an $\mathcal{I}$-density point of $U$. So $U$ is $\mathcal{I}-d$ open and $U \in \mathfrak{T}_{\mathcal{I}}$. Hence $\mathfrak{T}_d \subset \mathfrak{T}_{\mathcal{I}}$.
\end{proof}

\begin{rmrk}
It is shown in \cite{Density topologies} $\mathfrak{T}_d$ is finer than the natural topology on the real line. The inclusion $\mathfrak{T}_d \subset \mathfrak{T}_{\mathcal{I}}$ implies $\mathfrak{T}_{\mathcal{I}}$ is finer than the natural topology on the real line. However we give an alternative proof that $\mathcal{I}$-density topology is finer than usual topology on $\mathbb{R}$.
\end{rmrk}

\begin{thm}
The family $\mathfrak{T}_{\mathcal{I}}$ is a topology on the real line finer than the natural topology $\mathfrak{T}_U$.
\end{thm}
\begin{proof}
Let us take an open set $U$ in $\mathfrak{T}_U$. Since any $\mathfrak{T}_U$-open set in $\mathbb{R}$ can be written as countable union of disjoint open intervals so without any loss of generality let $U$ be an open interval $(a,b)$ where $a,b \in \mathbb{R}$ and $a<b$. We are to prove that $U$ is $\mathcal{I}-d$ open. Clearly $U$ is Lebesgue measurable. Now given any point $p$ in $U$ suppose $\{J_n\}_{n \in \mathbb{N}}$ be a sequence of closed interval about $p$ such that $\mathscr{S}(J_n)\in \mathcal{F}(\mathcal{I})$. Then there exists $n_0 \in \mathbb{N}$ such that for $n>n_0$ and $n \in \mathscr{S}(J_n)$ we have $J_n \subset U$. So for $n>n_0$ and $n \in \mathscr{S}(J_n)$, $x_n=\frac{m(J_n \cap U)}{m(J_n)}=1$. Therefore, $\{k\in \mathbb{N}: x_k=1\}\supset \mathscr{S}(J_n) \cap (\mathbb{N} \setminus \{1,2, \cdots ,n_o\})$. Thus, $\{k\in \mathbb{N}: x_k=1\} \in \mathcal{F}(\mathcal{I})$. So, $A_x=(1,\infty)$ and $\mathcal{I}-d_{-}(p,U)=\inf A_x=1$. Hence, $U$ is $\mathcal{I}-d$ open. Since any set that is open in natural topology $\mathfrak{T}_U$ on $\mathbb{R}$ is also $\mathcal{I}-d$ open so $\mathfrak{T}_{\mathcal{I}}$ is finer than $\mathfrak{T}_U$.
\end{proof}

\begin{dfn}
A set $F \subset \mathbb{R}$ is said to be $\mathcal{I}-d$ closed if $F^c$ is $\mathcal{I}-d$ open.
\end{dfn}

\begin{dfn}
A point $x \in \mathbb{R}$ is a $\mathcal{I}-d$ limit point of a set $E \subset \mathbb{R}$ (not necessarily measurable) if and only if $\mathcal{I}-d^{-}(x,E)>0$ where instead of taking measure $m$ outer measure $m^{*}$ is taken.
\end{dfn}

\begin{thm} In the space $(\mathbb{R}, \mathfrak{T}_\mathcal{I})$ given any Lebesgue measurable set $E \subset \mathbb{R}$, $m(E)=0$ if and only if $E$ is $\mathcal{I}-d$ closed and discrete.
\end{thm}
\begin{proof}Necessity: Let $m(E)=0$. Then for any $p \in \mathbb{R}$ and $\{I_n\}_{n \in \mathbb{N}}$ a sequence of intervals about $p$ such that $\mathscr{S}(I_n)\in \mathcal{F}(\mathcal{I})$. Take $x_n=\frac{m(I_n \cap E)}{m(I_n)}$. Then $x_n=0$ $\forall n \in \mathbb{N}$. So, $B_x =\{b\in \mathbb{R} : \{k:x_k > b\}\notin \mathcal{I}\}=(-\infty,0)$. Thus, $\mathcal{I}-d^{-}(p,E)=sup{B_x}=0$. Hence, $p$ is not a $\mathcal{I}-d$ limit point of $E$. So, $E$ has no $\mathcal{I}-d$ limit points. Therefore, $E$ is $\mathcal{I}-d$ closed and discrete. 

Sufficiency: Let $E$ be $\mathcal{I}-d$ closed and discrete. Then $E$ has no $\mathcal{I}-d$ limit points and so $\mathcal{I}-d^{-}(x,E)=0$ $\forall x \in \mathbb{R}$. Thus
$\mathcal{I}-d(x,E)=0$ $\forall x \in \mathbb{R}$. But, by Lebesgue $\mathcal{I}$-density theorem, $\mathcal{I}-d(x,E)=1$ for almost all $x \in E$. Therefore, $m(E)=0$.
\end{proof}

\begin{rmrk}
Although $\mathbb{Q}$ is neither open nor closed in ($\mathbb{R},\mathfrak{T}_U$) but since $m(\mathbb{Q})=0$ so by Theorem 4.8 it is $\mathcal{I}-d$ closed in $(\mathbb{R},\mathfrak{T}_\mathcal{I})$. The natural question arises that does there exist a subset of $\mathbb{R}$ which is neither $\mathcal{I}-d$ open nor $\mathcal{I}-d$ closed. In the following example we have shown that such sets do exists in $(\mathbb{R},\mathfrak{T}_\mathcal{I})$.
\end{rmrk}

\begin{xmpl}There exists a subset of $\mathbb{R}$ which is neither $\mathcal{I}-d$ open nor $\mathcal{I}-d$ closed. Here we are giving a construction of a collection of such sets in $\mathbb{R}$. Let us take an open interval $I=(x_1,x_2)$ where $x_1,x_2 \in \mathbb{Q}$ and $x_1 < x_2$. Since $I$ is open in $(\mathbb{R},\mathfrak{T}_\mathcal{I})$ it is $\mathcal{I}-d$ open. Now, let $b=\frac{(x_1+x_2)}{2}$. Then, $b$ is the center of $I$ and $b \in \mathbb{Q}$. Take $J=[b-\frac{|x_2-x_1|}{4},b+\frac{|x_2-x_1|}{4}]$. Then $J \subset I$. Define, $I'=I\setminus (J \cap \mathbb{Q}^{c})$. We claim that $I'$ is neither $\mathcal{I}-d$ open nor $\mathcal{I}-d$ closed. Let $\{I_k\}_{k\in \mathbb{N}}$ be a sequence of intervals about $b \in I'$ such that $\mathscr{S}(I_k)\in \mathcal{F}(\mathcal{I})$. Take, $x_k=\frac{m(I_k \cap I')}{m(I_k)}$. For large $k \in \mathscr{S}(I_k)$, $(I_k \cap I') \subset \mathbb{Q} $. Thus $m(I_k \cap I')=0$. Thus, $B_x =\{b\in \mathbb{R} : \{k:x_k > b\}\notin \mathcal{I}\}=(-\infty,0)$. Therefore, $\mathcal{I}-d^{-}(b,I')=\sup B_x =0$. Thus, $\mathcal{I}-d_{-}(b,I')=0$ and so $b$ is not an $\mathcal{I}$-density point of $I'$. Hence, $I'$ is not $\mathcal{I}-d$ open.

Now, to show $I'$ is not $\mathcal{I}-d$ closed we are to show $(I')^c$ is not $\mathcal{I}-d$ open. We see, $(I')^c=(-\infty,x_1] \cup (J \cap \mathbb{Q}^c) \cup [x_2,\infty)$. Let $\{J_k\}_{k \in \mathbb{N}}$ be a sequence of closed intervals about the point $x_1$ where we choose $J_k=[x_1-\frac{1}{2^{k+1}},x_1+\frac{1}{2^{k+1}}]$ $\forall k \in \mathbb{N}$. Take, $z_k=\frac{m(J_k \cap (I')^c)}{m(J_k)}$ where $0<m(J_k)=\frac{1}{2^k}<\frac{1}{k}$ $\forall k$. So $\mathscr{S}(J_k) = \mathbb{N} \in \mathcal{F}(\mathcal{I})$. Then, $m(J_k \cap (I')^c)<m(J_k)$ $\forall k$ implies $z_k<1$ $\forall k$. Then $\exists k_0 \in \mathbb{N}$ such that $z_k=\frac{1}{2}$ $\forall k>k_0$. Thus, if $B_z =\{b\in \mathbb{R} : \{k:z_k > b\}\notin \mathcal{I}\}=(-\infty, \frac{1}{2})$ then $\sup B_z= \mathcal{I}-d^{-}(x_1,(I')^{c})<1$. Therefore, $x_1$ is not an $\mathcal{I}$-density point of $(I')^c$. So, $(I')^c$ is not $\mathcal{I}-d$ open.
\end{xmpl}

\section{$\mathcal{I}$-approximate continuity}\label{section 5}

The notion of approximate continuity introduced by A. Denjoy is connected with the notion of Lebesgue density point. Since classical Lebesgue density point has been generalized to $\mathcal{I}$-density point subsequently in this section we have obtained the notion of  $\mathcal{I}$-approximate continuity.

\begin{dfn} [\cite{BRUCKNER}]
A function $f:\mathbb{R} \to \mathbb{R}$ is called  $\mathcal{I}$-approximately continuous at $x_0 \in \mathbb{R}$ if there exists a set $E_{x_0} \in \mathcal{L}$ such that $\mathcal{I}-d(x_0,E_{x_0})=1$ and $f|_{E_{x_0}}$ is continuous at $x_0$.
\end{dfn}
If $f$ is $\mathcal{I}$-approximately continuous at every point of $\mathbb{R}$ then we simply say $f$ is $\mathcal{I}$-approximately continuous. We use the notation $\mathcal{I}-\mathbb{A} \mathbb{C}$ to denote $\mathcal{I}$-approximately continuous. 

Now we prove the following results with suitable modification of classical proofs.

\begin{thm} If $f,g:\mathbb{R}\rightarrow{\mathbb{R}}$ be $\mathcal{I}-\mathbb{A} \mathbb{C}$ at $x_0$, then the functions $f+g, f\cdot g$ and $a \cdot g$ for any $a \in \mathbb{R}$ are $\mathcal{I}-\mathbb{A} \mathbb{C}$ at $x_0$. If $g(x)\neq0$  for any $x \in (x_0 - \delta, x_0 + \delta)$ where $\delta >0$ then $\frac{1}{g}$ is  $\mathcal{I}-\mathbb{A} \mathbb{C}$ at $x_0$.
\end{thm}
\begin{proof}
At first we show if $\mathcal{I}-d(x_0,A)=1$ and $\mathcal{I}-d(x_0,B)=1$ then $\mathcal{I}-d(x_0,A \cap B)=1$. Let $\{I_k\}_{k \in \mathbb{N}}$ be a sequence of closed intervals about $x_0$ such that $\mathscr{S}(I_k) \in \mathcal{F}(\mathcal{I})$. Since $m(A \cup B)=m(A)+m(B)-m(A \cap B)$. Then for $k \in \mathscr{S}(I_k)$ we have 
\begin{equation*}
    \frac{m((A \cup B)\cap I_k)}{m(I_k)}=\frac{m(A \cap I_k)}{m(I_k)}+\frac{m(B \cap I_k)}{m(I_k)}-\frac{m((A \cap B) \cap I_k)}{m(I_k)}
\end{equation*}
Taking $\mathcal{I}-\lim$ on both sides we obtain
\begin{equation*}
    \mathcal{I}-d(x_0,A \cap B)=\mathcal{I}-d(x_0,A)+\mathcal{I}-d(x_0,B)-\mathcal{I}-d(x_0,A \cup B)=1
\end{equation*}
Now since $f$ and $g$ are $\mathcal{I}-\mathbb{A}\mathbb{C}$ at $x_0$ so there exists two sets $E_f$ and $E_g$ in $\mathbb{R}$ such that $x_0$ is an $\mathcal{I}$-density point of both $E_f$ and $E_g$ and hence $\mathcal{I}-d(x_0, E_f \cap E_g)=1$. Also $f|_{E_f}$ and $g|_{E_g}$ are continuous at $x_0$. So,
\begin{equation*}
    (f+g)|_{E_f \cap E_g}=f|_{E_f \cap E_g}+g|_{E_f \cap E_g}
\end{equation*}
Hence, $(f+g)$ is $\mathcal{I}-\mathbb{A}\mathbb{C}$ at $x_0$.
Again, 
\begin{equation*}
    (f\cdot g)|_{E_f \cap E_g}=f|_{E_f \cap E_g} \cdot g|_{E_f \cap E_g}
\end{equation*}
Hence, $(f \cdot g)$ is $\mathcal{I}-\mathbb{A}\mathbb{C}$ at $x_0$. Similarly for any $a \in \mathbb{R}$,  $(a \cdot f)$ is $\mathcal{I}-\mathbb{A}\mathbb{C}$ at $x_0$.

Moreover, since $g(x)\neq0$  for any $x \in (x_0 - \delta, x_0 + \delta)$ where $\delta >0$, so $g|_{E_g \cap (x_0 - \delta, x_0 + \delta)} \neq 0$ and continuous at $x_0$. Then $(\frac{1}{g})|_{E_g \cap (x_0 - \delta, x_0 + \delta)}$ is continuous at $x_0$ and $x_0$ is an $\mathcal{I}$-density point of $E_g \cap (x_0 - \delta, x_0 + \delta)$. Hence $\frac{1}{g}$ is $\mathcal{I}-\mathbb{A}\mathbb{C}$ at $x_0$.
\end{proof}

\begin{thm}
If $f$ is $\mathcal{I}-\mathbb{A}\mathbb{C}$ at $x_0$ and $g$ is continuous at $f(x_0)$ then $(g \circ f)$ is $\mathcal{I}-\mathbb{A}\mathbb{C}$ at $x_0$.
\end{thm}
\begin{proof}
By hypothesis there exists $E_f$ a subset of $\mathbb{R}$ such that $\mathcal{I}-d(x_0, E_f)=1$ and $f|_{E_f}$ is continuous at $x_0$. Now $(g \circ f)|_{E_f}=g \circ f|_{E_f}$. Since composition of two continuous functions is continuous so $(g \circ f)|_{E_f}$ is continuous at $x_0$. Thus, $(g \circ f)$ is $\mathcal{I}-\mathbb{A}\mathbb{C}$ at $x_0$.
\end{proof}

We state here the Lusin's Theorem for our future purpose.
\begin{thm}[\cite{Oxtoby}] A real valued function $f$ on $\mathbb{R}$ is measurable if and only if for each $\epsilon>0$ there exists a set $E$ with $m(E)<\epsilon$ such that the restriction of $f$ to $\mathbb{R} \setminus E$ is continuous.

\end{thm}

\begin{thm}
A function $g:\mathbb{R} \rightarrow{\mathbb{R}}$ is measurable if and only if it is $\mathcal{I}-\mathbb{A}\mathbb{C}$ almost everywhere.
\end{thm}
\begin{proof}
Necessary part: Let $g$ be measurable. For $\epsilon >0$, by Lusin's Theorem, there exists a continuous function $\psi$ such that $m(\{x:g(x) \neq \psi(x)\})<\epsilon$. Let $E=\{x:g(x) \neq \psi(x)\}$. Since $E$ is measurable so $E^c$ is measurable. By Theorem, 3.2 almost every point of $E^c$ is a point of $\mathcal{I}-$density of $E^c$ and $g|_{E^c}=\psi$ is continuous. So $g$ is $\mathcal{I}-\mathbb{A}\mathbb{C}$ at almost every point of $E^c$. Thus $g$ is $\mathcal{I}-\mathbb{A}\mathbb{C}$ except on $E$ where outer measure of $E$ is less than $\epsilon$. So, $g$ is $\mathcal{I}-\mathbb{A}\mathbb{C}$ almost everywhere, since $\epsilon>0$ is arbitrary.

Sufficient part: Suppose $g$ is $\mathcal{I}-\mathbb{A}\mathbb{C}$ almost everywhere. To show $g$ is measurable. For $r \in \mathbb{R}$ let $E_r=\{x:g(x)<r\}$. It is sufficient to show that $E_r$ is measurable. Without any loss of generality let $E_r$ be uncountable. Let $B=\{x\in \mathbb{R}: g \ \mbox{is} \ \mathcal{I}-\mathbb{A}\mathbb{C} \ \mbox{at} \ x\}$. Then
\begin{equation*}
    E_r=(E_r \cap B) \cup (E_r \setminus B)
\end{equation*}
From hypothesis $m(\mathbb{R}\setminus B)=0$. Since $m$ is a complete measure so $E_r \setminus B \in \mathcal{L}$. It is enough if we can show $E_r \cap B \in \mathcal{L}$. Let $t \in E_r \cap B$. Since $t \in B$ so there exists a set $D_t \in \mathcal{L}$ such that $\mathcal{I}-d(t,D_t)=1$ and $f|_{D_t}$ is continuous at $t$. Without any loss of generality $D_t$ can be chosen inside $E_r \cap B$. Therefore 
\begin{equation*}
    E_r \cap B=\bigcup_{t \in E_r \cap B}D_t
\end{equation*}
If possible let $E_r \cap B$ be not measurable. Then there exists an Euclidean $F_\sigma$ set $P$ and Euclidean $G_\delta$ set $H$ such that $P \subset E_r \cap B \subset H$ and 
\begin{equation*}
    m(P)=m_{\star}(E_r \cap B)<m^{\star}(E_r \cap B)=m(H)
\end{equation*}
Thus $m(H \setminus P)>0$. By theorem 3.2 almost every point of $H \setminus P$ is a point of $\mathcal{I}$-density of $H \setminus P$. Since $m(H \setminus P)=m^{\star}((E_r \cap B) \setminus P)$, so $m^{\star}((E_r \cap B) \setminus P)>0$. There exists $t_0 \in (E_r \cap B) \setminus P \subset H \setminus P$ such that $\mathcal{I}-d(t_0, H \setminus P)=1$. Now $t_0 \in (E_r \cap B)$. So there exists set $D_{t_0} \subset E_r \cap B$ such that $\mathcal{I}-d(t_0, D_{t_0})=1$. We claim that $m(D_{t_0} \setminus P)>0$. For, if possible, let
\begin{equation}
    m(D_{t_0} \setminus P)=0
\end{equation}
Then $\mathcal{I}-d(t_0,D_{t_0} \setminus P)=0$. Now $H=D_{t_0} \cup (H \setminus D_{t_0})$. So, by Theorem 2.12,
\begin{equation}
    \mathcal{I}-d(t_0,H \setminus D_{t_0})=0, \ \mbox{since} \ \mathcal{I}-d(t_0,D_{t_0})=1 
\end{equation}
Here
\begin{equation*}
    H \setminus P=(D_{t_0} \setminus P) \cup ((H \setminus P) \setminus D_{t_0})
\end{equation*}
Now from (5.1) and (5.2) we have 
\begin{equation*}
    \mathcal{I}-d(t_0,H \setminus P)=\mathcal{I}-d(t_0,D_{t_0} \setminus P)+\mathcal{I}-d(t_0,(H \setminus P) \setminus D_{t_0})=0
\end{equation*}
This is a contradiction. Now $m(D_{t_0} \setminus P)>0$ implies $m_{\star}(D_{t_0} \setminus P)>0$. Then $m_{\star}((E_r \cap B) \setminus P)>0$. This contradicts to the fact that $m(P)=m_{\star} (E_{r} \cap B)$. Thus $E_r \cap B \in \mathcal{L}$.
\end{proof}

\begin{dfn}[\cite{natanson}]
The set of all continuous functions defined on interval $I$ is called as the null Baire class of functions. If the function $g(x)$ defined on $I$ is not in the null class but is representable in the form 
\begin{equation}
    g(x)=\lim_{n \rightarrow{\infty}}g_n(x)
\end{equation}
where all the functions $g_n(x)$ are continuous then $g(x)$ is said to be a function of the first Baire class. In general the functions of Baire class $m \in \mathbb{N}$ are functions which are not in any of the preceeding classes but can be represented as the limit of sequence of functions of Baire class $(m-1)$ as in (5.3).

In this way all the classes of functions with finite indices are defined. We denote these classes by $\mathcal{B}_0,\mathcal{B}_1,\dots,\mathcal{B}_m,\dots$
\end{dfn}

\begin{thm}[\cite{natanson}]
Let $I$ be a fixed interval and $g:I\rightarrow{\mathbb{R}}$ be a function of class not greater than $m_1$ and let $\psi$ be a function of class not greater than $m_2$ whose values lie in $I$. Then $(g \circ \psi)$ is a function of class $\leq m_1+m_2$.
\end{thm}

\begin{thm}[\cite{natanson}]
Let $I$ be a fixed interval. Then $g:I\rightarrow{\mathbb{R}}$ is a function of Baire class not greater than the first if and only if for arbitrary $\alpha \in \mathbb{R}$ the sets $C^{\alpha}=\{x:g(x)<\alpha\}$ and $C_{\alpha}=\{x:g(x)>\alpha\}$ are of type Euclidean $F_{\sigma}$.
\end{thm}

\begin{thm}
Given any fixed interval $I$ if $g:I \rightarrow{\mathbb{R}}$ is $\mathcal{I}-\mathbb{A}\mathbb{C}$ function, then $g$ belongs to first Baire class.
\end{thm}
\begin{proof}
Since $g$ is $\mathcal{I}-\mathbb{A}\mathbb{C}$ so by Theorem 5.5, $g$ is measurable. First let us take $g$ to be bounded. Then there exists a positive number $M$ such that $|g(x)|<M$ for $x \in I$. Now for $a \in I$ define
\begin{equation*}
    G(x)=\int_{a}^{x}g(t)dt
\end{equation*}
Then $G: I \rightarrow{\mathbb{R}}$ is a continuous function. We claim for each $r \in I$, 
\begin{equation}
    \lim_{k \rightarrow{0}}\frac{G(r+k)-G(r)}{k}=g(r)
\end{equation}
i.e. given any $\epsilon>0$ there exists $\delta>0$ such that $|\frac{1}{k}\int_{r}^{r+k}g(t)dt-g(r)|<\epsilon$ whenever $k<\delta$. Since $g$ is $\mathcal{I}-\mathbb{A}\mathbb{C}$ on $I$ so for $r \in I$ there exists $B_r \subset I$ such that $\mathcal{I}-d(r,B_r)=1$ and $g|_{B_r}$ is continuous at $r$. So for each $k>0$
\begin{equation}
\begin{split}
    \bigg |\frac{1}{k}\int_{r}^{r+k}g(t)dt-g(r)\bigg | & = \bigg |\frac{1}{k}\int_{r}^{r+k}(g(t)-g(r))dt \bigg |\\
    & \leq \frac{1}{k}\int_{r}^{r+k}|g(t)-g(r)|dt\\
    & = \frac{1}{k}\int_{[r,r+k]\cap B_r}|g(t)-g(r)|dt+\frac{1}{k}\int_{[r,r+k]\setminus B_r}|g(t)-g(r)|dt
\end{split}
\end{equation}
Now for given any $\epsilon>0$ we choose $\delta>0$ such that
\begin{enumerate}
    \item since $g|_{B_r}$ is continuous at $r$ so for $t \in B_r \cap (r-\delta,r+\delta)$ we have $|g(t)-g(r)|<\frac{\epsilon}{2}$
    \item since $\mathcal{I}-d(r,B_r)=1$ so $\mathcal{I}-d(r,B_r^{c})=0$. Hence for some $k<\delta$ we have $\frac{m([r,r+k]\setminus B_r)}{k}< \frac{\epsilon}{4M}$
\end{enumerate}
For $k<\delta$ from $(5.5)$ we obtain
\begin{equation}
    \begin{split}
        \bigg |\frac{1}{k}\int_{r}^{r+k}g(t)dt-g(r)\bigg | & \leq \frac{1}{k} \cdot \frac{\epsilon}{2}\cdot m([r,r+k])+\frac{1}{k}\cdot 2M \cdot m([r,r+k]\setminus B_r)\\
        & < \frac{1}{k} \cdot \frac{\epsilon}{2}\cdot k+\frac{1}{k}\cdot 2M \cdot \frac{\epsilon k}{4M}\\
        & = \epsilon
    \end{split}
\end{equation}
Similarly calculating for $k<0$ we obtain (5.4). Thus for each $r \in I$ we have
\begin{align*}
     g(r) &=\lim_{k \rightarrow{0}}\frac{G(r+k)-G(r)}{k} = \lim_{n \rightarrow{\infty}}\frac{G(r+\frac{1}{n})-G(r)}{\frac{1}{n}} = \lim_{n \rightarrow{\infty}} n\{G(r+\frac{1}{n})-G(r)\}
\end{align*}

Now let $G_n(r)=n\{G(r+\frac{1}{n})-G(r)\}$. Then $G_n$ is continuous since $G$ is continuous. Therefore $g$ is in first Baire class.

Now if $g:I \rightarrow{\mathbb{R}}$ is unbounded then let $h:\mathbb{R} \rightarrow{(0,1)}$ be a homeomorphism. So, $h$ and $h^{-1}$ are continuous. Also by Theorem 5.3, \ $h \circ g: I \rightarrow{(0,1)}$ is $\mathcal{I}-\mathbb{A}\mathbb{C}$ and $(h \circ g)$ is bounded. So by the first part $(h \circ g)$ is in first Baire class. Now $g=h^{-1} \circ (h \circ g)$. Hence by Theorem 5.7, $g$ is in first Baire class. 

\end{proof}

The next lemma is based on the idea presented in \cite{Density topologies}(Theorem 3.1) and the condition presented in this lemma will be called the condition $(J_2)$ of J. M. Jedrzejewski.

\begin{lma}
Let $\{G_n\}_{n \in \mathbb{N}}$ be any decreasing sequence of Lebesgue measurable sets such that for some $x_0 \in \mathbb{R}$, $\mathcal{I}-d(x_0,G_n)=1$ \ $\forall n \in \mathbb{N}$. Then there exists a decreasing sequence $\{s_n\}_{n \in \mathbb{N}}$ of positive real numbers converging to zero such that
\begin{equation*}
    A_{x_0}=\bigcup_{n=1}^{\infty}(G_n \setminus (x_0-s_n,x_0+s_n)) \ \mbox{and} \ \mathcal{I}-d(x_0,A_{x_0})=1
\end{equation*}
\end{lma}

\begin{proof}
Let $\{\delta_n\}_{n \in \mathbb{N}}$ be a strictly decreasing sequence such that $0<\delta_n <1$ $\forall n \in \mathbb{N}$ and $\delta_n \rightarrow{0}$ as $n \rightarrow{\infty}$. Now since $\mathcal{I}-d(x_0,G_n)=1$ $\forall n \in \mathbb{N}$ so for any sequence of closed intervals $\{I_k\}_{k \in \mathbb{N}}$ about $x_0$ such that $\mathscr{S}(I_k) \in \mathcal{F}(\mathcal{I})$ we have $\mathcal{I}-\lim_{k}\frac{m(G_n \cap I_k)}{m(I_k)}=1$. So for given any $\epsilon>0$ and for each $n \in \mathbb{N}$,
\begin{equation*}
    C_{\epsilon}^{(n)}=\{k:\frac{m(G_n \cap I_k)}{m(I_k)}>1-\epsilon\} \in \mathcal{F}(\mathcal{I})
\end{equation*}
Now for $\epsilon=\delta_n$ there exists $k_n \in \mathbb{N}$ such that for $k \geq k_n$ and $k \in C_{\delta_n}^{(n)}$, 
\begin{equation*}
    \frac{m(G_n \cap I_k)}{m(I_k)}> 1-\delta_{n}
\end{equation*}
We choose $k_n$'s so that $\{k_n\}_{n \in \mathbb{N}}$ is increasing and the sequence $\{m(I_{k_n})\}_{n \in \mathbb{N}}$ is decreasing. Thus consider a subsequence $\{I_{k_n}\}_{k_n \in C_{\delta_n}^{(n)}}$ of the sequence $\{I_k\}_{k \in C_{\delta_n}^{(n)}}$ and put
\begin{equation*}
    s_n=\delta_n m(I_{k_{n+1}}) \ \mbox{for} \ n \in \mathbb{N} \ \mbox{and} \ k_{n+1} \in C_{\delta_n}^{(n)}
\end{equation*}
Since $\delta_n \rightarrow{0}$ and $m(I_{k_{n+1}})<\frac{1}{k_{n+1}}$ so $s_n \rightarrow{0}$ as $n \rightarrow{\infty}$. Since $\delta_n$ is decreasing and $m(I_{k_{n}})$ is decreasing, $s_n$ is decreasing. Without any loss of generality we can assume that $m(I_k)$ is decreasing for $k \in C_{\delta_n}^{(n)}$. For $\delta >0$ there exists $n_0 \in \mathbb{N}$ such that $3 \delta_n < \delta$ for $n > n_0$. Moreover there exists $l_1 \in \mathbb{N}$ such that $m(I_k)<m(I_{k_{n_0+1}})$ for $k > l_1$ and $k \in C_{\delta_n}^{(n)}$. Now fix $k>l_1$ and $k \in C_{\delta_n}^{(n)}$. So there exists $n_1 > n_0$ such that 
\begin{equation*}
    m(I_{k_{n_1+1}})\leq m(I_{k}) < m(I_{k_{n_1}})
\end{equation*}
Since $\{m(I_k)\}_{k \in C_{\delta_n}^{(n)}}$ is decreasing sequence so $k > k_{n_1}$. Thus for fixed $n=n_1$ we have
\begin{equation}
    \begin{split}
        \frac{m((G_{n_1} \setminus (x_0-s_{n_1},x_0+s_{n_1}))\cap I_k)}{m(I_k)}
        & = \frac{m((G_{n_1} \cap I_k)\setminus (x_0-s_{n_1},x_0+s_{n_1}))}{m(I_k)}\\
        & \geq \frac{m(G_{n_1} \cap I_k)-2s_{n_1}}{m(I_k)}\\
        & = \frac{m(G_{n_1} \cap I_k)}{m(I_k)}-\frac{2s_{n_1}}{m(I_k)}\\
        & > 1-\delta_{n_1}-\frac{2 \delta_{n_1}m(I_{k_{n_1+1}})}{m(I_k)}\\
        & > 1- \delta_{n_1} - 2 \delta_{n_1}\\
        & = 1- 3\delta_{n_1}\\
        & > 1-\delta
    \end{split}
\end{equation}
So since for all $k \in \mathbb{N}$
\begin{equation*}
    \frac{m(A_{x_0}\cap I_k)}{m(I_k)} \geq \frac{m((G_{n_1} \setminus (x_0-s_{n_1},x_0+s_{n_1}))\cap I_k)}{m(I_k)}
\end{equation*}
We have 
\begin{equation*}
    \{k:\frac{m((G_{n_1} \setminus (x_0-s_{n_1},x_0+s_{n_1}))\cap I_k)}{m(I_k)}>1-\delta\} \subset \{k:\frac{m(A_{x_0}\cap I_k)}{m(I_k)}>1-\delta\}
\end{equation*}
Moreover since $\mathcal{I}$ is an admissible ideal, so
\begin{equation*}
    \{k:\frac{m((G_{n_1} \setminus (x_0-s_{n_1},x_0+s_{n_1}))\cap I_k)}{m(I_k)}>1-\delta\} \supset C_{\delta_n}^{(n)} \cap (\mathbb{N}\setminus \{1,2,\cdots,l_1\}) \in \mathcal{F}(\mathcal{I})
\end{equation*}
Hence,$\{k: \frac{m(A_{x_0}\cap I_k)}{m(I_k)}>1-\delta\}\in \mathcal{F}(\mathcal{I})$. Therefore, $\mathcal{I}-\lim_{k} \frac{m(A_{x_0}\cap I_k)}{m(I_k)}=1$. Thus $\mathcal{I}-d(x_0,A_{x_0})=1$.
\end{proof}

\begin{thm}
Given any fixed interval $I$,\ $g:I \rightarrow{\mathbb{R}}$ is $\mathcal{I}-\mathbb{A}\mathbb{C}$ function if and only if for each $\mu \in \mathbb{R}$ both the sets $C^{\mu}=\{x:g(x)<\mu\}$ and $C_{\mu}=\{x:g(x)>\mu\}$ belongs to the topology $\mathfrak{T}_\mathcal{I}$.
\end{thm}
\begin{proof}Necessity: Let $g$ be $\mathcal{I}-\mathbb{A}\mathbb{C}$. Then by Theorem 5.9, $g$ is in first Baire class. So by Theorem 5.8, for each $\mu \in \mathbb{R}$, \ $C^{\mu}$ and $C_{\mu}$ are of type Euclidean $F_{\sigma}$. So,  $C^{\mu},C_{\mu}\in \mathcal{L}$. Now we are to show for each $x \in C^{\mu}$, \ $\mathcal{I}-d(x,C^{\mu})=1$.

Let us fix $\mu \in \mathbb{R}$ and let us take $x_0 \in C^{\mu}$. Then $g(x_0)<\mu$. So, $\mu - g(x_0)>0$. Since $g$ is $\mathcal{I}-\mathbb{A}\mathbb{C}$ at $x_0$ so there exists $E \in \mathcal{L}$ such that $\mathcal{I}-d(x_0,E)=1$ and $g|_{E}$ is continuous at $x_0$. Hence, given any $\epsilon >0$ there exists $\delta>0$ such that $x \in (x_0-\delta,x_0+\delta)\cap E$ implies $g(x_0)-\epsilon<g(x)<g(x_0)+\epsilon$. In particular if we choose $\epsilon_{0}=\frac{\mu-g(x_0)}{M}$ for some $M \in \mathbb{N}$ and $M>1$, then $g(x_0)=\mu-M \epsilon_{0}$. So for suitably chosen $\delta_{0}>0$ and for $x \in (x_0-\delta_{0},x_0+\delta_{0})\cap E$ we have 
\begin{equation*}
    g(x)<g(x_0)+\epsilon_{0}=\mu-M\epsilon_{0}+\epsilon_{0}<\mu
\end{equation*}
Thus, $(x_0-\delta_{0},x_0+\delta_{0})\cap E \subset C^{\mu}$. Since, $x_0$ is an $\mathcal{I}$-density point of $(x_0-\delta_{0},x_0+\delta_{0})$ and $E$ so it is $\mathcal{I}$-density point of $(x_0-\delta_{0},x_0+\delta_{0})\cap E$. Therefore, $\mathcal{I}-d(x_0,C^{\mu})=1$.

Sufficiency: Let $x_0 \in I$. Without any loss of generality, we choose $x_0$ in $I$ without being the end points of $I$. Let $\{\epsilon_n\}_{n \in \mathbb{N}}$ be a decreasing sequence of positive real numbers converging to zero. For each $n \in \mathbb{N}$ let $A_n=\{x:g(x)<g(x_0)+\epsilon_{n}\}$ and $B_n=\{x:g(x)>g(x_0)-\epsilon_{n}\}$. By hypothesis $A_n,B_n \in \mathfrak{T}_{\mathcal{I}}$. Let $C_n=A_n \cap B_n$ \ $\forall n \in \mathbb{N}$. Then $C_n=\{x:|g(x)-g(x_0)|<\epsilon_{n}\}$. We observe $C_n \in \mathfrak{T}_{\mathcal{I}}$. Since $x_0 \in C_n$ so $\mathcal{I}-d(x_0,C_n)=1$ $\forall n \in \mathbb{N}$. Since $\{C_n\}_{n \in \mathbb{N}}$ is a decreasing sequence of measurable sets so by lemma 5.10 there exists a strictly decreasing sequence $\{s_n\}_{n \in \mathbb{N}}$ of positive real numbers converging to zero such that
\begin{equation*}
    A_{x_0}=\bigcup_{n=1}^{\infty}(C_n \setminus (x_0-s_n,x_0+s_n)) \ \mbox{and} \ \mathcal{I}-d(x_0,A_{x_0})=1
\end{equation*}
Then $A_{x_0} \in \mathcal{L}$. Now we are to show $g|_{A_{x_0}}$ is continuous at $x_0$. For fixed $\epsilon>0$ there exists $n_0 \in \mathbb{N}$ such that $\epsilon_{n} < \epsilon$ \ $\forall n > n_0$. Now if $x \in A_{x_0} \cap (x_0-s_{n_0},x_0+s_{n_0})$ then $x \in \bigcup_{n=n_0+1}^{\infty}(C_n \setminus (x_0-s_n,x_0+s_n))$. So there exists $n_1 > n_0$ such that $x \in C_{n_1}$. Let us choose $\delta = s_{n_0}$. Then for $x \in A_{x_0} \cap (x_0-\delta,x_0+\delta)$ we have $x \in C_{n_1}$ i.e. $|g(x)-g(x_0)|<\epsilon_{n_1}< \epsilon$. Therefore $g|_{A_{x_0}}$ is continuous at $x_0$. Hence $g$ is $\mathcal{I}-\mathbb{A}\mathbb{C}$ at $x_0$.

\end{proof}

\begin{dfn} [cf.\cite{Hejduk 2017}]
A function $g:\mathbb{R}\rightarrow{\mathbb{R}}$ is $\mathcal{I}$-approximately upper semi-continuous at a point $x_0 \in \mathbb{R}$ if for every $\alpha > g(x_0)$ there exists a set $E_{x_0} \in \mathcal{L}$ such that $\mathcal{I}-d(x_0,E_{x_0})=1$ and $g(x)<\alpha$ for every $x \in E_{x_0}$.

Moreover, $g$ is $\mathcal{I}$-approximately upper semi-continuous if it is $\mathcal{I}$-approximately upper semi-continuous at every point $x \in \mathbb{R}$. Similarly we define $\mathcal{I}$-approximately lower semi-continuity.
\end{dfn}

\begin{thm} A function $g:\mathbb{R}\rightarrow{\mathbb{R}}$ is $\mathcal{I}-\mathbb{A}\mathbb{C}$ if and only if it is $\mathcal{I}$-approximately upper and $\mathcal{I}$-approximately lower semi-continuous.
\end{thm}
\begin{proof}
Necessity: Let $g$ be $\mathcal{I}-\mathbb{A}\mathbb{C}$ at $x_0 \in \mathbb{R}$. So there exists $E \in \mathcal{L}$ such that $\mathcal{I}-d(x_0,E)=1$ and $g|_E$ is continuous at $x_0$. So given any $\epsilon >0$ there exists $\delta>0$ such that whenever $x \in (x_0 -\delta, x_0 + \delta) \cap E$ then $g(x_0)-\epsilon<g(x)<g(x_0)+\epsilon$. Now for every $c\in \mathbb{R}$ and $c>g(x_0)$ choose $\epsilon>0$ such that $g(x_0)+\epsilon<c$. For this $\epsilon>0$, we choose $\delta>0$ such that for every $x \in (x_0 -\delta, x_0 + \delta) \cap E$ we have $g(x)<g(x_0)+\epsilon<c$. Moreover $x_0$ is an $\mathcal{I}$-density point of $(x_0 -\delta, x_0 + \delta) \cap E$. Thus $g$ is $\mathcal{I}$-approximately upper semi-continuous at $x_0$. Since choice of $x_0 \in \mathbb{R}$ is arbitrary hence $g$ is $\mathcal{I}$-approximately upper semi-continuous at every $x \in \mathbb{R}$. Similarly it can be shown that $g$ is $\mathcal{I}$-approximately lower semi-continuous at every $x \in \mathbb{R}$.

Sufficiency: Let $g$ be $\mathcal{I}$-approximately upper and $\mathcal{I}$-approximately lower semi-continuous. For any $\alpha \in \mathbb{R}$ let $C^{\alpha}=\{x \in \mathbb{R}: g(x)<\alpha\}$. Now take $x_0 \in C^{\alpha}$. Then $g(x_0)<\alpha$. Since $g$ is $\mathcal{I}$-approximately upper semi-continuous at $x_0$ so there exists $E_{x_0} \in \mathcal{L}$ such that $\mathcal{I}-d(x_0,E_{x_0})=1$ and $\forall$ $x \in E_{x_0}$, $g(x)<\alpha$. Let us take $\widehat{E}_{x_0}=\{x_0\} \cup E_{x_0}$. Then $\widehat{E}_{x_0} \in \mathcal{L}$. Now define
\begin{equation*}
    V_{x_0}=\{y\in \widehat{E}_{x_0}: \mathcal{I}-d(y,\widehat{E}_{x_0})=1\}
\end{equation*}
Then $V_{x_0}$ is $\mathcal{I}-d$ open and $V_{x_0}\in \mathfrak{T}_{\mathcal{I}}$. Moreover
\begin{equation*}
   y \in V_{x_0} \implies y \in \widehat{E}_{x_0} \implies g(y)<\alpha \implies y \in C^{\alpha} 
\end{equation*} Thus $V_{x_0} \subset C^{\alpha}$. Since choice of $x_0$ is arbitrary so $V_x \subset C^{\alpha}$ for all $x \in C^{\alpha}$. Therefore, $C^{\alpha}=\bigcup_{x \in C^{\alpha}}V_x$ where $V_{x}\in \mathfrak{T}_{\mathcal{I}}$. Consequently, $C^{\alpha} \in \mathfrak{T}_{\mathcal{I}}$. 

In a similar approach we can show for any $\beta \in \mathbb{R}$, $C_{\beta}=\{x \in \mathbb{R}: g(x)>\beta\} \in \mathfrak{T}_{\mathcal{I}}$. Thus by Theorem 5.11 it can be concluded $g$ is $\mathcal{I}-\mathbb{A}\mathbb{C}$.
\end{proof}

We now proceed to prove the main result of this section.

\begin{thm} A function $g:(\mathbb{R},\mathfrak{T}_\mathcal{I}) \to (\mathbb{R},\mathfrak{T}_U)$ is continuous if and only if $g$ is $\mathcal{I}-\mathbb{A}\mathbb{C}$ at every $x\in \mathbb{R}$
\end{thm}

\begin{proof}
Necessity: Let $g:(\mathbb{R},\mathfrak{T}_\mathcal{I}) \to (\mathbb{R},\mathfrak{T}_U)$ is continuous at $x_0$. So given any $\mathfrak{T}_U$-open set $V$ containing $g(x_0)$ there exists $\mathcal{I}-d$ open set $U$ containing $x_0$ such that $x_0 \in U \subset g^{-1}(V)$. Since $U$ is $\mathcal{I}-d$ open set and $x_0 \in U$, $\mathcal{I}-d(x_0,U)=1$ and so $g|_{U}$ is continuous at $x_0$. Hence $g$ is $\mathcal{I}-\mathbb{A}\mathbb{C}$ at $x_0$. Since choice of $x_0$ is arbitrary. So $g$ is $\mathcal{I}-\mathbb{A}\mathbb{C}$ at every $x$.

Sufficiency: Let $g$ be $\mathcal{I}-\mathbb{A}\mathbb{C}$. Then by Theorem 5.11 for any $\mu \in \mathbb{R}$ we have $C^{\mu}=\{x:g(x)<\mu\}$ and $C_{\mu}=\{x:g(x)>\mu\}$ where both $C^{\mu}$ and $C_{\mu}$ are in $\mathfrak{T}_\mathcal{I}$. Then let $g$ be $\mathcal{I}-\mathbb{A}\mathbb{C}$ at $x_0$ for some $x_0 \in \mathbb{R}$. Let $V$ is an open set in $(\mathbb{R},\mathfrak{T}_U)$ containing $g(x_0)$. Without any loss of generality let $V=(g(x_0)-\epsilon',g(x_0)+\epsilon)$ for some $\epsilon, \epsilon' >0$.  We are to show there exists a set $U \in \mathfrak{T}_\mathcal{I}$ containing $x_0$ such that $g(U) \subset V$. Let $C^{\star}=\{x:g(x)<g(x_0)+\epsilon\}$ and $C_{\star}=\{x:g(x)>g(x_0)-\epsilon'\}$. Then 
\begin{equation*}
    C^{\star} \cap C_{\star}=\{x:g(x_0)-\epsilon'<g(x)<g(x_0)+\epsilon\}
\end{equation*}
Let $C^{\star} \cap C_{\star}=U$. Then $U \in \mathfrak{T}_\mathcal{I}$. Observe that $x_0 \in U$. Now $x \in U$ implies $g(x) \in (g(x_0)-\epsilon',g(x_0)+\epsilon)$. Therefore $g(U) \subset V$. Hence $g$ is continuous at $x_0$. 
\end{proof}

\section{Lusin-Menchoff Theorem}\label{section 6}

The Lusin-Menchoff theorem plays a vital role in proving complete regularity of density topology \cite{Zahorski}. In this paper since we attempt to prove complete regularity of $\mathcal{I}$-density topology so we try to prove analogue of Lusin-Menchoff theorem for $\mathcal{I}$-density.

\begin{dfn}[\cite{kechris}]
A topologiocal space is called Polish if it is seperable and completely metrizable space.
\end{dfn}

\begin{xmpl}
$(\mathbb{R},\mathfrak{T}_U)$ is a Polish space.
\end{xmpl}

\begin{dfn}[\cite{kechris}]
A topological space $X$ is perfect if all its points are limit points or equivalently it contains no isolated points.
\end{dfn}

If $P$ is a subset of a topological space $X$ then $P$ is perfect in $X$ if $P$ is closed and perfect in its relative topology. The following theorem is known as Cantor-Bendixon theorem.
\begin{thm}[\cite{kechris}]
Let $X$ be a Polish space. Then $X$ can be written uniquely as $X=P \cup C$ where $P$ is a perfect subset of $X$ and $C$ is countable open.
\end{thm}

The above result holds good if we take any closed set instead of $X$. Now we state the Perfect set Theorem for Borel sets.
\begin{thm}[\cite{kechris}]
Let $X$ be a Polish space and $A \subset X$ be Borel. Then either $A$ is countable or else it contains a Cantor set.
\end{thm}

Now we will prove some lemmas which will be needed later in this section.

\begin{lma} Let $B$ be a Borel set. Then for $x\in B$ such that $\mathcal{I}-d(x,B)=1$ there exists a $\mathfrak{T}_U$ perfect set $P$ such that $x\in P \subset B$ 
\end{lma}
\begin{proof}For $x \in B$, $\mathcal{I}-d(x,B)=1$ implies $\mathcal{I}-\lim$ $b_n=1$ where $b_n=\frac{m(I_n \cap B)}{m(I_n)}$, $\{I_n\}_{n \in \mathbb{N}}$ being a sequence of closed intervals about $x$ such that $\mathscr{S}(I_n)\in \mathcal{F}(\mathcal{I})$. Given $\epsilon>0$ let $A_{\epsilon}=\{n:|b_n-1|<\epsilon\}$ then $A_{\epsilon} \in \mathcal{F}(\mathcal{I})$. For $n \in A_{\epsilon}$,
\begin{equation}
    b_n>1-\epsilon\\
    \implies m(I_n\cap B)>(1-\epsilon)m(I_n)\\
    \implies m(I_n\cap B)>0
\end{equation}
Let us take a sequence $\{J_n\}_{n \in \mathbb{N}}=\{[c_n,d_n]\}_{n \in \mathbb{N}}$ of pairwise disjoint intervals such that $dist(x,J_n)\rightarrow{0}$ as $n \rightarrow{\infty}$ and without any loss of generality assume $m(J_n \cap B)>0$ $\forall n \in A_{\epsilon}$. So, $J_n \cap B$ is not countable $\forall n \in A_{\epsilon}$. Since both $J_n$ and $B$ are Borel sets so, $J_n \cap B$ is Borel. Now since $(\mathbb{R},\mathfrak{T}_U)$ is a Polish space by Theorem 6.5, $\forall n \in A_{\epsilon}$, there exists a $\mathfrak{T}_U$-perfect set $P_n$ such that $P_n \subset J_n \cap B$. Since $J_n$'s are pairwise disjoint. Therefore, $\{P_n\}_{n \in A_{\epsilon}}$ is a collection of pairwise disjoint $\mathfrak{T}_U$-perfect set.

Now let $P=\{x\} \cup (\bigcup_{n \in A_{\epsilon}}P_n)$. Then $x \in P \subset B$.\\ 
We claim that $P$ is $\mathfrak{T}_U$-perfect set.

First we show $P$ has no isolated points. Now since for $i \in A_{\epsilon}$  each $P_i$ is $\mathfrak{T}_U$-perfect, so $P
_i$ has no isolated point. Hence $\bigcup_{i \in A_{\epsilon}}P_i$ has no isolated point.
Now to show $x$ is not an isolated point of $P$. Let $N(x)$ be any open neighbourhood about $x$. Then for some $n_0 \in A_{\epsilon}$, $J_{n_0} \cap (N(x)\setminus \{x\}) \neq \phi$. Then for $n_0^{'}>n_0$ and $n_0^{'}\in A_{\epsilon}$ there exists $\mathfrak{T}_U$-perfect set $P_{n_0^{'}}$ such that $P_{n_0^{'}} \cap (N(x)\setminus \{x\})$ is nonempty. Hence $P \cap (N(x)\setminus \{x\})$ is nonempty. So, $x$ is not an isolated point of $P$. Therefore, $P$ has no isolated points.

Next we show $P$ is $\mathfrak{T}_U$-closed. Let $\{s_n\}_{n \in \mathbb{N}}$ be a sequence in $P$ such that $s_n \rightarrow{s}$. We are to show $s \in P$. We have the following two cases:\\
Case(i) Let there are finitely many $s_n$ in each $P_i$ for $i \in A_{\epsilon}$ then without any loss of generality we may assume $s_i \in P_i$ for each $i \in A_{\epsilon}$. Then claim: $s=x$. 
For large $i\in A_{\epsilon}$,
\begin{equation}
    |s-x|\leq|s-s_i|+|s_i-x|\leq|s-s_i|+dist(x,P_{i'})
\end{equation}
Here $i'$ in the subscript of $P_{i'}$ is the immediate predecessor of $i$ in $A_{\epsilon}$.
Also since $dist(x,J_n)\rightarrow{0}$ as $n \rightarrow{\infty}$. So, $dist(x,P_i)\rightarrow{0}$ as $i \rightarrow{\infty}$. 
So, as $i \rightarrow{\infty}$ from 6.2 we can conclude $s=x$. Hence, $s \in P$.\\
Case(ii) If at least one of $P_n$ say $P_i$ contains infinitely many of $s_n$ then suppose there exists a subsequence $\{s_{n_k}\}_{k \in \mathbb{N}}$ of $\{s_n\}_{n \in \mathbb{N}}$ such that $\{s_{n_k}\} \subset P_i$. Since $s_{n_k} \rightarrow{x}$ and $P_i$ is $\mathfrak{T}_U$-perfect so, $s \in P_i$. Therefore, $s \in P$. \\
Hence, $P$ is $\mathfrak{T}_U$-closed. Consequently, $P$ is $\mathfrak{T}_U$-perfect.
\end{proof}

\begin{lma} Let $B$ be a Borel set. Then for every countable set  $C$ such that $cl(C) \subset B$ and $\mathcal{I}-d(x,B)=1 \forall x \in C$ there exists a $\mathfrak{T}_U$ perfect set $P$ such that $C \subset P \subset B$. Here $cl(C)$ stands for $\mathfrak{T}_U$-closure of C.
\end{lma}
\begin{proof}Let us take $C=\{x_i: i \in \mathbb{N}\} \subset B$. Now put $B_i=B \cap [x_i-\frac{1}{2^i},x_i+\frac{1}{2^i}]$ for $i \in \mathbb{N}$. Then $B_i$ is a Borel set containing $x_i$ for each $i$. We claim that $\mathcal{I}-d(x_i,B_i)=1$. Now since $\mathcal{I}-d(x_i,B)=1$. So, $\mathcal{I}-\lim$ $b_n=1$ where $b_n=\frac{m(I_n \cap B)}{m(I_n)}$, $\{I_n\}_{n \in \mathbb{N}}$ being a sequence of closed intervals about $x_i$ such that $\mathscr{S}(I_n)\in \mathcal{F}(\mathcal{I})$. Given $\epsilon>0$ let $A_{\epsilon}=\{n:|b_n-1|<\epsilon\}$ then $A_{\epsilon} \in \mathcal{F}(\mathcal{I})$. Since $B_i \subset B$, $\exists n_0 \in \mathbb{N}$ such that $\forall n>n_0$ and $n \in A_{\epsilon}$ we have $m((B\setminus B_i)\cap I_n)=0$. Therefore, $\forall n>n_0$ and $n \in A_{\epsilon}$, $\frac{m(B\cap I_n)}{m(I_n)}=\frac{m(B_i \cap I_n)}{m(I_n)}$. Hence $\mathcal{I}-d(x_i,B)=\mathcal{I}-d(x_i,B_i)$. So, $\mathcal{I}-d(x_i,B_i)=1$ for each $i \in \mathbb{N}$. By Lemma 6.6 there exists a $\mathfrak{T}_U$-perfect set $P_i$ such that $x_i \in P_i \subset B_i$ for each $i \in \mathbb{N}$. 

Now let $P=cl(C) \cup (\bigcup_{i \in \mathbb{N}}P_i)$. Then clearly $C \subset P \subset B$. 

Claim: $P$ is $\mathfrak{T}_U$-perfect.

It is clear that $\bigcup_{i \in \mathbb{N}}P_i$ has no isolated points. Now if $x \in C$. Then $x=x_i$ for some $i$ and $x_i \in P_i$ where $P_i$ is $\mathfrak{T}_U$-perfect. So, $x$ is not an isolated point of $P$. Again if $x\in cl(C)\setminus C$ then there exists a sequence $\{z_n\}_{n\in \mathbb{N}}\in C$ such that $z_n \rightarrow{x}$. Then any open neighbourhood $N(x)$ about $x$ contains some $z_i \neq x$. Consequently, $x$ is not an isolated point of $P$.

So, $P$ has no isolated points. 

Next we are to show $P$ is $\mathfrak{T}_U$-closed. Let $\{s_n\}_{n \in \mathbb{N}}$ be a sequence in $P$ such that $s_n \rightarrow{s}$. To show $s \in P$. We have the following three cases:\\
Case(i): If $cl(C)$ contains infinitely many of $s_n$ then suppose $\{s_{n_k}\}_{k \in \mathbb{N}}$ be a subsequence of $\{s_n\}_{n \in \mathbb{N}}$ such that $\{s_{n_k}\}_{k \in \mathbb{N}} \subset cl(C)$. Since $s_{n_k} \rightarrow{x}$ and $cl(C)$ is $\mathfrak{T}_U$-closed then $s \in cl(C)$. So, $s \in P$.\\
Case(ii): If atleast one of $P_i$ contains infinitely many of $s_n$ then suppose $\{s_{n_k}\}_{k \in \mathbb{N}}$ be a subsequence of $\{s_n\}_{n \in \mathbb{N}}$ such that $\{s_{n_k}\}_{k \in \mathbb{N}} \subset P_i$. Since $s_{n_k} \rightarrow{x}$ and $P_i$ is $\mathfrak{T}_U$-perfect so, $s \in P_i$. Hence $s \in P$.\\
Case(iii): Let there be finitely many $s_n$ in each $P_i$ for $i \in \mathbb{N}$ then without any loss of generality we may assume $s_i \in P_i$ for each $i \in \mathbb{N}$. Since $m(P_i)\leq m(B_i)<\frac{1}{2^{i-1}}$ for each $i \in \mathbb{N}$, therefore, $|x_i-s_i|<\frac{1}{2^{i-1}}$ for each $i \in \mathbb{N}$. Now for each $k \in \mathbb{N}$
\begin{equation*}
    |x_k-s|\leq|x_k-s_k|+|s_k-s|<\frac{1}{2^{k-1}}+|s_k-s|
\end{equation*}
From the above inequality it can be concluded $x_k \rightarrow{s}$ as $k \rightarrow{\infty}$. Therefore, $s \in cl(C)$, since $x_i \in C$ $\forall i$. Thus, $s \in P$.
\end{proof}

\begin{lma} Let $H$ be a Lebesgue measurable set. Then for every $\mathfrak{T}_U$ closed subset $Z$ of $H$ such that $\mathcal{I}-d(x,H)=1$ $\forall x \in Z$ there exists a $\mathfrak{T}_U$-perfect set $P$ such that $Z \subset P \subset H$.
\end{lma}
\begin{proof}Since $H$ is a measurable subset of $\mathbb{R}$ so there exists an Euclidean $F_\sigma$ set $A \subset H$ such that $m(H\setminus A)=0$. For $x\in \mathbb{R}$, let $\{I_k\}_{k \in \mathbb{N}}$ be a sequence of closed intervals about $x$ such that $\mathscr{S}(I_k)\in \mathcal{F}(\mathcal{I})$. Let $h_k=\frac{m(H \cap I_k)}{m(I_k)}$ and $a_k=\frac{m(A \cap I_k)}{m(I_k)}$. Then,

\begin{align*}
    h_k
         = \frac{m(H \cap I_{k})}{m(I_{k})}  = \frac{m((A \cup (H \setminus A)) \cap I_k)}{m(I_k)} & = \frac{m(A \cap I_{k})}{m(I_{k})}+\frac{m((H\setminus A) \cap I_{k})}{m(I_{k})}\\
        & = \frac{m(A \cap I_{k})}{m(I_{k})}
         = a_k
\end{align*}
Hence, 
\begin{equation*}
    \mathcal{I}-d(x,H)=\mathcal{I}-\lim h_k=\mathcal{I}-\lim a_k=\mathcal{I}-d(x,A)
\end{equation*}
Therefore, $\mathcal{I}-d(x,A)=1$ $\forall x\in Z$. Since both $A$ and $Z$ are Borel sets so $B=A \cup Z$ is Borel and $Z \subset B \subset H$. Also $\mathcal{I}-d(x,B)=1$ $\forall x\in Z$, since $A \subset B$. Since $(\mathbb{R},\mathfrak{T}_U)$ is a Polish space, so by Theorem 6.4, $Z=P_1 \cup C$ where $P_1$ is $\mathfrak{T}_U$-perfect and $C$ is countable set. Now since $Z$ is $\mathfrak{T}_U$-closed we have $cl(C)\subset Z \subset B$ and $\mathcal{I}-d(x,B)=1$ $\forall x \in C$. By Lemma 6.7 there exists a $\mathfrak{T}_U$-perfect set $P_2$ such that $C \subset P_2 \subset B$. Therefore $P_1 \cup C \subset P_1 \cup P_2 \subset B$. Take $P=P_1 \cup P_2$. Then $P$ is $\mathfrak{T}_U$-perfect and $Z\subset P \subset B \subset H$.
\end{proof}

Now we prove an analogue of Lusin-Menchoff Theorem for $\mathcal{I}$-density.

\begin{thm} Let $H$ be a measurable set. Then for every $\mathfrak{T}_U$ closed set $Z$ such that $Z \subset H$ and $\mathcal{I}-d(x,H)=1$ $\forall x \in Z$ there exists a $\mathfrak{T}_U$ perfect set $P$ such that $Z \subset P \subset H$ and $\mathcal{I}-d(x,P)=1$ $\forall x \in Z$
\end{thm}
\begin{proof}By hypothesis and Lemma 6.8 there exists $\mathfrak{T}_U$-perfect set $K$ such that $Z \subset K \subset H$. Now define,
\begin{center}
    $H_n=\{z \in H: \frac{1}{n+1}< dist(z,Z) \leq \frac{1}{n}\}$ for $n \in \mathbb{N}$
\end{center}
and let
\begin{center}
    $H_0=\{z \in H: dist(z,Z)>1\}$
\end{center}
Then, $H=Z \cup (\bigcup_{n=0}^{\infty}H_n)$. Without any loss of generality let us assume that each $H_n$ is nonempty. Since $dist$ function is continuous so $H_n$'s are measurable for each $n \in \mathbb{N} \cup \{0\}$. So for every $n\in \mathbb{N} \cup \{0\}$ we can find a closed set $F_n \subset H_n$ such that $m(H_n \setminus F_n)<\frac{1}{2^{n+1}}$. By Cantor Bendixon theorem since every closed set can be expressed as a union of a perfect set and a countable set, for each $n$ there exists $\mathfrak{T}_U$-perfect set $P_n \subset F_n \subset H_n$ such that $m(H_n \setminus P_n)<\frac{1}{2^{n+1}}$. Put,
\begin{center}
    $P=K \cup (\bigcup_{n=1}^{\infty}P_n)$
\end{center}
Then $P$ is nonempty $\mathfrak{T}_U$-perfect set such that $Z \subset P \subset H$.\\

Now we are to show that $\mathcal{I}-d(x,P)=1$ $\forall x \in Z$.

For $x \in Z$, by hypothesis $\mathcal{I}-d(x,H)=1$. Let $\{I_k\}_{k \in \mathbb{N}}$ be a sequence of closed intervals about $x$ such that $\mathscr{S}(I_k) \in \mathcal{F(\mathcal{I})}$. Then 
    $\mathcal{I}-\lim_{k}\frac{m(I_k \cap H)}{m(I_k)}=1$
Since $H=Z \cup (\bigcup_{n=0}^{\infty}H_n)$. So
\begin{align*}
    H \setminus P  & = (Z \setminus P)\cup (\{\bigcup_{n=0}^{\infty}H_n\}\setminus P)
   = \{\bigcup_{n=0}^{\infty}H_n\}\setminus P \\
  & = \bigcup_{n=0}^{\infty}(H_n \setminus P)
  = \bigcup_{n=0}^{\infty}H_n \setminus (K \cup \bigcup_{m=1}^{\infty}P_m)\\ 
   & = \bigcup_{n=0}^{\infty}((H_n \setminus K) \cap (H_n \setminus \bigcup_{m=1}^{\infty}P_m))\\
 & = \bigcup_{n=0}^{\infty}((H_n \setminus K) \cap (\bigcap_{m=1}^{\infty}(H_n \setminus P_m)))
\end{align*}
From here $k$ will be chosen in $\mathscr{S}(I_k)$. So we have
\begin{equation}
    I_k \cap (H \setminus P) = \bigcup_{n=0}^{\infty}(I_k \cap (H_n \setminus K) \cap (\bigcap_{m=1}^{\infty}(H_n \setminus P_m)))
\end{equation}
Now for a fixed $k$ there are two possibilities:
\begin{enumerate}
        \item $\exists \ n_k \in \mathbb{N}$ such that 
$I_k \cap H_n = \phi$ for $n<n_k$ but $I_k \cap H_{n_k} \neq \phi$
        \item $I_k \cap H_n = \phi$ $\forall$ $n$. In this case we put $n_k=\infty$.
\end{enumerate}

For case $(2)$ the R.H.S. in (6.3) is empty. So $m(I_k \cap (H \setminus P))=0$. Therefore, $\frac{m(I_k \cap H)}{m(I_k)}=\frac{m(I_k \cap P)}{m(I_k)}$. Hence,\begin{center}
    $\mathcal{I}-d(x,P)=\mathcal{I}-\lim_{k} \frac{m(I_k \cap P)}{m(I_k)}=\mathcal{I}-\lim_{k} \frac{m(I_k \cap H)}{m(I_k)}=1$
\end{center}

For case $(1)$ from (6.3) we have
\begin{equation}
    I_k \cap (H \setminus P)  = \bigcup_{n=n_k}^{\infty}(I_k \cap (H_n \setminus K) \cap (\bigcap_{m=1}^{\infty}(H_n \setminus P_m)))
\end{equation}

\begin{equation}
\begin{split}
    m(I_k \cap (H \setminus P))
    & \leq \sum_{n=n_k}^{\infty}m(I_k \cap (H_n \setminus K) \cap (\bigcap_{m=1}^{\infty}(H_n \setminus P_m)))\\
    & \leq \sum_{n=n_k}^{\infty}m(H_n \setminus P_n)\\
    & < \sum_{n=n_k}^{\infty} \frac{1}{2^{n+1}} = \frac{1}{2^{n_k}}
\end{split}
\end{equation}

Now we will consider the following two subcases:\\
Subcase $(i)$: Let us assume for each $k$, $n_k < \infty$. We claim that as $k \rightarrow{\infty}$ then $n_k \rightarrow{\infty}$. To show given $N \in \mathbb{N}$ there exists $k_0 \in \mathbb{N}$ such that if $k>k_0$ then $n_k>N$.\\

For given any large $N \in \mathbb{N}$ let $k_0=N+1$. If $k>k_0$ then $m(I_k)<\frac{1}{k}<\frac{1}{k_0}$. Also $I_k \cap H_{n_k} \neq \phi$. Let $y \in I_k \cap H_{n_k}$. 
\begin{equation}
    \mbox{Since} \ x,y \in I_k \ \mbox{so} \  |x-y|<m(I_k)<\frac{1}{k_0}
\end{equation}  
Moreover since
\begin{equation}
     x \in Z \ \mbox{and} \ y \in H_{n_k}, \ |x-y| \geq dist(H_{n_k},Z) > \frac{1}{n_k+1}
\end{equation}
From equation (6.6) and (6.7) we have $\frac{1}{n_k+1}<\frac{1}{k_0}=\frac{1}{N+1}$ which implies that $n_k>N$. Also note that $m(I_k)>\frac{1}{n_k+1}$ since $\frac{1}{n_k+1}<|x-y|<m(I_k)$ by (6.6).\\
Now for $k>k_0$,
\begin{equation}
    \begin{split}
    \frac{m(I_k \cap H)}{m(I_k)}
    & = \frac{m(I_k \cap P)}{m(I_k)}+\frac{m(I_k \cap (H \setminus P))}{m(I_k)}\\
    & < \frac{m(I_k \cap P)}{m(I_k)}+\frac{n_k+1}{2^{n_k}}
\end{split}
\end{equation}
Therefore, $\mathcal{I}-\lim_k \frac{m(I_k \cap H)}{m(I_k)} < \mathcal{I}-\lim_k \frac{m(I_k \cap P)}{m(I_k)}$ by (6.8). So $\mathcal{I}-d(x,P)>1$. Hence $\mathcal{I}-d(x,P)=1$.

Subcase $(ii)$: For infinitely many $k$ let $n_k=\infty$. So let there exists a subsequence $\{k_l\}$ of $\{k\}$ such that $n_{k_l}=\infty$ and  $k_l \rightarrow{\infty}$ as $l \rightarrow{\infty}$. So, $I_{k_l} \cap H_n = \phi$ $\forall n$. Hence $m(I_{k_l} \cap (H \setminus P))=0$. So, $\frac{m(I_{k_l} \cap H)}{m(I_{k_l})}=\frac{m(I_{k_l} \cap P)}{m(I_{k_l})}$. Thus,
\begin{center}
    $\mathcal{I}-\lim_l \frac{m(I_{k_l} \cap H)}{m(I_{k_l})}=\mathcal{I}-\lim_l \frac{m(I_{k_l} \cap P)}{m(I_{k_l})}=1$
\end{center}
Hence, the Theorem is proved.
\end{proof}

\section{Some separation axioms}\label{section 7}
The purpose of this section is to provide some information about separation axioms for the space $(\mathbb{R},\mathfrak{T}_\mathcal{I})$. Since by Theorem 4.5, $\mathfrak{T}_U \subset \mathfrak{T}_\mathcal{I}$ we obtain immediately the following result.

\begin{prop}
The space $(\mathbb{R},\mathfrak{T}_\mathcal{I})$ is a Hausdorff space.
\end{prop}

In the next theorem we obtain a bounded $\mathcal{I}-\mathbb{A}\mathbb{C}$ function. Given any two sets $A$ and $B$ we use the notation $A \subset\bullet \ B$ to mean $A \subset B$ and $\mathcal{I}-d(x,B)=1$ $\forall x \in A$ (cf. \cite{BRUCKNER}).

\begin{thm} Let $H$ be a set of type Euclidean $F_\sigma$ such that $\mathcal{I}-d(x,H)=1$ $\forall x \in H$. Then there exists an $\mathcal{I}-\mathbb{A}\mathbb{C}$ function $g:\mathbb{R} \rightarrow{\mathbb{R}}$ such that 
\begin{center}
    $(1)$ $0<g(x)\leq 1$ for $x \in H$\\
    $(2)$ $g(x)=0$   \quad \ \ for $x \notin H$
\end{center}
\end{thm}
\begin{proof}If $H=\phi$ then $g(x)=0$ $\forall x \in \mathbb{R}$ and so $g$ is $\mathcal{I}-\mathbb{A}\mathbb{C}$. Let $H$ be a nonempty Euclidean $F_{\sigma}$ set. So, $H=\bigcup_{n=1}^{\infty}K_n$ where each $K_n$ is nonempty $\mathfrak{T}_U-$closed set. Now we construct a family of $\mathfrak{T}_U$ closed sets $\{Q_\beta: \beta \in \mathbb{R} \ \mbox{and} \ \beta \geq 1\}$ such that $Q_{\beta_1} \subset\bullet \  Q_{\beta_2}$ if $\beta_1 < \beta_2$ and $H=\bigcup_{\beta \geq 1}Q_{\beta}$.

Let $Q_1=K_1$. Since $Q_1 \subset H$ where $H$ is measurable and $Q_1$ is $\mathfrak{T}_U$-closed set and $\mathcal{I}-d(x,H)=1$ $\forall x \in Q_1$. So by Theorem 6.9 $\exists$ $\mathfrak{T}_U$ closed set $B_2$ such that $Q_1 \subset B_2 \subset H$ and $\mathcal{I}-d(x,B_2)=1$ $\forall x \in Q_1$. So $Q_1 \subset\bullet \ B_2 \subset\bullet \ H$. Now take $Q_2=K_2 \cup B_2$. Then $Q_1 \subset\bullet \ Q_2 \subset\bullet \ H$. We proceed inductively. Suppose $\exists$ $\mathfrak{T}_U$ closed set $Q_n$ satisfying $Q_{n-1} \subset\bullet \ Q_n \subset\bullet \ H$ and $K_n \subset Q_n$. Then by Theorem 6.9, $\exists$ $\mathfrak{T}_U$ closed set $B_{n+1}$ such that $Q_{n} \subset\bullet \ B_{n+1} \subset\bullet \ H$. Let $Q_{n+1}=K_{n+1} \cup B_{n+1}$. Then $Q_{n} \subset\bullet \ Q_{n+1} \subset\bullet \ H$ and $K_{n+1} \subset Q_{n+1}$. By induction we obtain the collection $\{Q_n\}_{n \in \mathbb{N}}$ such that $K_n \subset Q_n$ $\forall n \in \mathbb{N}$ and $Q_n \subset H$ $\forall n \in \mathbb{N}$. Therefore, 
\begin{equation}
    H=\bigcup_{n \in \mathbb{N}}Q_n
\end{equation}.

Now by Theorem 6.9 for each $l \in \mathbb{N} \cup \{0\}$ and $n \geq 2^l$ we define a $\mathfrak{T}_U-$closed set $Q_{\frac{n}{2^l}}$ such that
\begin{equation}
    Q_{\frac{n}{2^l}} \subset\bullet \ Q_{\frac{(n+1)}{2^l}}
\end{equation}. So we have the following:

For $l=0$ we get $Q_1 \subset\bullet \ Q_2 \subset\bullet \ Q_3 \subset\bullet \ \cdots$

For $l=1$ we get $Q_1 \subset\bullet \ Q_{\frac{3}{2}} \subset\bullet \ Q_2 \subset\bullet \ Q_{\frac{5}{2}} \subset\bullet \ Q_3 \subset\bullet \ \cdots$

For $l=2$ we get $Q_1 \subset\bullet \ Q_{\frac{5}{4}} \subset\bullet \ Q_{\frac{3}{2}} \subset\bullet \ Q_{\frac{7}{4}} \subset\bullet \ Q_2 \subset\bullet \ Q_{\frac{9}{4}} \subset\bullet \ Q_{\frac{5}{2}} \subset\bullet \ \cdots$

and so on. 

Suppose for fixed $l_0$ we choose $Q_{\frac{n}{2^{l_0}}}$ $\forall n \geq 2^{l_0}$ such that $Q_{\frac{n}{2^{l_0}}} \subset\bullet \ Q_{\frac{(n+1)}{2^{l_0}}}$. Since $Q_{\frac{n}{2^{l_0}}}=Q_{\frac{2n}{2^{l_0+1}}}$. So by (7.2) and Theorem 6.9 we have $Q_{\frac{2n}{2^{l_0+1}}} \subset\bullet \ Q_{\frac{2n+1}{2^{l_0+1}}}$ and $Q_{\frac{2n+1}{2^{l_0+1}}} \subset\bullet \ Q_{\frac{2n+2}{2^{l_0+1}}}$. 

Therefore, $Q_{\frac{n}{2^{l_0}}} \subset\bullet \ Q_{\frac{2n+1}{2^{l_0+1}}} \subset\bullet \ Q_{\frac{(n+1)}{2^{l_0}}}$. In particular we get 
\begin{equation*}
    Q_1 \subset\bullet \ \cdots \subset\bullet \ Q_{\frac{9}{8}} \subset\bullet \ \cdots \subset\bullet \  Q_{\frac{5}{4}} \subset\bullet \ \cdots \subset\bullet \ Q_{\frac{3}{2}}  \subset\bullet \ \cdots  \subset\bullet \ Q_{\frac{7}{4}} \subset\bullet \ \cdots \subset\bullet \ Q_{\frac{15}{8}} \cdots \subset\bullet \ Q_2 \cdots
\end{equation*} 

For each real number $\beta \geq 1$ we define
\begin{equation*}
    Q_{\beta}=\bigcap_{\frac{n}{2^l} \geq \beta}Q_{\frac{n}{2^l}}
\end{equation*}.

Moreover since each $Q_{\frac{n}{2^l}}$ is $\mathfrak{T}_U-$closed so $Q_{\beta}$ is $\mathfrak{T}_U-$closed. Now if $\beta_1 < \beta_2$ we can choose sufficiently large $l_0$ so that for some $n_0 \in \mathbb{N}$ we have $2^{l_0} \beta_1 < n_0 < (n_0+1) < 2^{l_0} \beta_2$. Observe that $Q_{\frac{(n_0+1)}{2^{l_0}}} \subset Q_{\frac{n}{2^{l}}}$ $\forall \frac{n}{2^{l}} \geq \beta_2$. Hence $Q_{\frac{(n_0+1)}{2^{l_0}}} \subset \bigcap_{\frac{n}{2^l} \geq \beta_2}Q_{\frac{n}{2^{l}}}=Q_{\beta_2}$. So, $Q_{\beta_1} \subset Q_{\frac{n_0}{2^{l_0}}} \subset\bullet \ Q_{\frac{(n_0+1)}{2^{l_0}}} \subset Q_{\beta_2}$. Consequently, $Q_{\beta_1} \subset\bullet \ Q_{\beta_2}$. Thus 
\begin{equation*}
    H=\bigcup_{\beta \geq 1}Q_{\beta}
\end{equation*}

We define $g: \mathbb{R} \rightarrow{\mathbb{R}}$ where
\begin{equation}
    g(x) = \left\{ \begin{array}{rcl}
\frac{1}{\inf \{\beta: x \in Q_{\beta}\}} & \mbox{if}
& x \in H \\ 0 & \mbox{if} & x \notin H
\end{array}\right.
\end{equation}
Since $\beta \geq 1$ so $g(x) \leq 1$ $\forall x \in \mathbb{R}$. Now take $x \in H$. From (7.1) we see $\exists$ some $n_0 \in \mathbb{N}$ such that $x \in Q_{n_0}$. So, $\inf \{\beta: x \in Q_{\beta}\} \leq n_0$ which means $g(x) \geq \frac{1}{n_0} > 0$. So, $0<g(x) \leq 1$ $\forall x \in H$.

Next we are to prove $g$ is $\mathcal{I}-\mathbb{A}\mathbb{C}$. For that we first show, $g$ is continuous on $H^c$. Let $x_0 \in H^c$. So by (7.1) $x_0 \in Q_n^c$ $\forall n$. Say $x_0 \in Q_N^c$ for some $N \in \mathbb{N}$. Since $Q_N$ is $\mathfrak{T}_U$ closed $\exists$ $\delta>0$ such that $Q_N \cap (x_0 - \delta, x_0 + \delta)=\phi$. Now since $Q_{\beta_1} \subset Q_{\beta_2}$ for $\beta_1 < \beta_2$ therefore for $\beta \leq N$ we get $Q_\beta \cap (x_0 - \delta, x_0 + \delta)=\phi$. Thus if $\beta \leq N$ and $x \in (x_0 -\delta, x_0 + \delta)$ then $x \in Q_\beta^c$. Thus $\inf \{\beta: x \in Q_\beta\} \geq N$ and so $g(x) \leq \frac{1}{N}$ for $x \in (x_0 - \delta, x_0 + \delta)$. Since choice of $N$ is arbitrary so $g(x_0)=0$ $\forall x_0 \in H^c$. So, $g$ is continuous on $H^c$.

Now we prove $g$ is upper semi-continuous at any $x_0 \in H$. Let $g(x_0)=\frac{1}{\lambda^{'}}$. Then for $\lambda < \lambda^{'}$ we observe $x_0 \notin Q_{\lambda}$. Since $Q_{\lambda}$ is $\mathfrak{T}_U$ closed so for sufficiently small $\delta>0$ we have $(x_0-\delta,x_0+\delta) \subset Q_{\lambda}^{c}$. Thus for any $x \in (x_0-\delta,x_0+\delta)$ since $\inf\{\beta:x \in Q_{\beta}\}>\lambda$ we have  $g(x)-g(x_0)< \frac{1}{\lambda}-\frac{1}{\lambda^{'}}$. So we are done.

Now to show $g$ is $\mathcal{I}$-approximately lower semi-continuous at points $x \in H$. Let $x_0 \in H$ and suppose $g(x_0)=\frac{1}{\lambda}$. For any $\alpha < g(x_0)$ let $C_{\alpha}=\{x: g(x)>\alpha\}$. It is enough to show $\mathcal{I}-d(x_0, C_{\alpha})=1$. Since $\alpha <\frac{1}{\lambda}$ there exists $\delta>0$ such that $\alpha<\frac{1}{\lambda + 2 \delta}< \frac{1}{\lambda}$. Now we observe $\lambda = \inf\{\beta: x_0 \in Q_{\beta}\}$. So clearly $x_0 \in Q_{\lambda + \delta}$. From the properties of the family $\{Q_{\beta}: \beta \geq 1\}$ we have $Q_{\lambda + \delta} \subset\bullet \ Q_{\lambda + 2\delta}$. Therefore $\mathcal{I}-d(x_0, Q_{\lambda + 2\delta})=1$. We claim that $Q_{\lambda + 2\delta} \subset C_{\alpha}$. For any $x \in Q_{\lambda + 2\delta}$ implies $\inf \{\beta: x \in Q_{\beta}\} \leq \lambda + 2 \delta$ which means $g(x) \geq \frac{1}{\lambda + 2 \delta}$. Since $\frac{1}{\lambda + 2 \delta} > \alpha$ so $g(x)>\alpha$. Consequently $x \in C_{\alpha}$. Hence $Q_{\lambda + 2\delta} \subset C_{\alpha}$. So, $\mathcal{I}-d(x_0,C_{\alpha})=1$. Hence, $g$ is $\mathcal{I}$-approximately lower semi-continuous.

 Thus $g$ is $\mathcal{I}-\mathbb{A}\mathbb{C}$ function.
\end{proof}

We now show $(\mathbb{R},\mathfrak{T}_\mathcal{I})$ is completely regular. To prove this theorem we need the following lemma.

\begin{lma} Let $P_1, P_2, G$ be pairwise disjoint subsets of $\mathbb{R}$ such that 
\begin{center}
\begin{enumerate}[label=(\roman*)]
    \item $P_1 \cup P_2 \cup G =\mathbb{R}$
    \item $P_1 \cup G$ and $P_2 \cup G$ are $\mathcal{I}-d$ open and of type Euclidean $F_\sigma$
\end{enumerate}    
\end{center}
Then there exists an $\mathcal{I}-\mathbb{A}\mathbb{C}$ function $g$ such that
\begin{center}
\begin{enumerate}[label=(\roman*)]
    \item $g(x)=0$ for $x \in P_1$
    \item $0<g(x)<1$ for $x \in G$
    \item $g(x)=1$ for $x \in P_2$
\end{enumerate}    
\end{center}
\end{lma}
\begin{proof}Since, $P_1 \cup G$ and $P_2 \cup G$ both are Euclidean $F_\sigma$ and also $(P_1 \cup G)^c=P_2$ and $(P_2 \cup G)^c=P_1$. So, by Theorem 7.2 there exists two $\mathcal{I}-\mathbb{A}\mathbb{C}$ functions $g_1$ and $g_2$ such that
\begin{center}
   $0<g_1(x)\leq 1$ for $x \in P_2 \cup G$ and $g_1(x)=0$ for $x \in P_1$   \\
   $0<g_2(x)\leq 1$ for $x \in P_1 \cup G$ and $g_2(x)=0$ for $x \in P_2$   
   
\end{center}

Now take, $\psi: (\mathbb{R} \times \mathbb{R})\setminus \{(0,0)\} \rightarrow{[0,1]}$  where $\psi (x_1, x_2)=\frac{|x_1|}{|x_1|+|x_2|}$. Then, 
\begin{center}
    $\psi(0,x_2)=0$ for $x_2 \neq 0$    \\
    $\psi(x_1,0)=1$ for $x_1 \neq 0$  \\
    $0<\psi(x_1,x_2)<1$ for $x_1 \neq 0, x_2 \neq 0$
\end{center}
Then $\psi$ is continuous except at $\{(0,0)\}$. We consider, $g(x)=\psi(g_1(x),g_2(x))$. Since modulus function is continuous, so by Theorem 5.3, $|g_1(x)|$ and $|g_2(x)|$ are $\mathcal{I}-\mathbb{A}\mathbb{C}$. Moreover $|g_1(x)|+|g_2(x)| \neq 0$ for all $x$. Hence, by Theorem 5.2, $g$ is $\mathcal{I}-\mathbb{A}\mathbb{C}$. 

Then for $x \in P_1$, $g(x)=\psi(0,g_2(x))=0$ since $g_2(x)\neq 0$ and for $x \in P_2$, $g(x)=\psi(g_1(x),0)=1$ since  $g_1(x)\neq 0$. Finally, for $x \in G$, $g_1(x) \neq 0$ and $g_2(x) \neq 0$. So, $g(x)=\frac{|g_1(x)|}{|g_1(x)|+|g_2(x)|}$. Thus $0<g(x)<1$ for $x \in G$.

\end{proof}

\begin{thm} The space $(\mathbb{R},\mathfrak{T}_\mathcal{I})$ is completely regular.
\end{thm}

\begin{proof} Let $F$ be $\mathcal{I}-d$ closed set in $\mathbb{R}$ and $p_0 \notin F$. Since every $\mathcal{I}-d$ open set is measurable, $F$ is measurable. Let H be an Euclidean $G_\delta$-set such that $F \subset H$, $m(H\setminus F)=0$ and $p_0 \notin H$. Let us put $P_1=H, P_2=\{p_0\}$ and $G=\mathbb{R} \setminus (P_1 \cup P_2)$. Then, $P_1 \cup G=\mathbb{R} \setminus \{p_0\}=(-\infty,p_0)\cup(p_0,\infty)$. Since each of $(-\infty,p_0)$ and $(p_0,\infty)$ are Euclidean $F_\sigma$-set so their union is Euclidean $F_\sigma$-set. Moreover, $(-\infty,p_0)$ and $(p_0,\infty)$ are $\mathfrak{T}_{U}$ open so $\mathcal{I}-d$ open. Again, $P_2 \cup G=\mathbb{R} \setminus H$ is Euclidean $F_\sigma$-set, $H$ being an Euclidean $G_\delta$-set. We observe $\mathbb{R} \setminus H=(\mathbb{R}\setminus F)\setminus (H \setminus F)$. Since $\mathbb{R}\setminus F$ is $\mathcal{I}-d$ open and $m(H\setminus F)=0$ so $\mathbb{R} \setminus H$ is $\mathcal{I}-d$ open. By Lemma 7.3, there exists an $\mathcal{I}-\mathbb{A}\mathbb{C}$ function $g: \mathbb{R} \rightarrow{\mathbb{R}}$ such that
\begin{center}
\begin{enumerate}
    \item $g(x)=0$ for $x \in H$ and $H \supset F$
    \item $0<g(x)<1$ for $x \in G$
    \item $g(x)=1$ for $x=p_0$
\end{enumerate}
\end{center} 
Therefore, $g(x)=0$ on $F$ and $g(p_0)=1$. So, by Theorem 5.14, $g$ is a continuous function on $(\mathbb{R},\mathfrak{T}_\mathcal{I})$. Hence, $(\mathbb{R},\mathfrak{T}_\mathcal{I})$ is completely regular.
\end{proof}

\begin{thm}
The space $(\mathbb{R},\mathfrak{T}_\mathcal{I})$ is not normal.
\end{thm}
\begin{proof}
Since $(\mathbb{R},\mathfrak{T}_d)$ is not normal and by Theorem 4.3. $\mathfrak{T}_\mathcal{I}$ is finer than $\mathfrak{T}_d$ so $(\mathbb{R},\mathfrak{T}_{\mathcal{I}})$ is not normal.
\end{proof}

\section*{Acknowledgements}
\textit{The second author is thankful to The Council of Scientific and Industrial Research (CSIR), Government of India, for giving the award of Junior Research Fellowship (File no. 09/025(0277)/2019-EMR-I) during the tenure of preparation of this research paper.}

\end{document}